\documentclass[11pt]{amsart}

\newtheorem{thm}{Theorem}[section]
\newtheorem*{thm*}{Theorem}

\newtheorem{cor}[thm]{Corollary}

\newtheorem{lem}[thm]{Lemma}
\newtheorem*{lem*}{Lemma}
\newtheorem{mainthm}{Theorem}
\newtheorem*{mainthm*}{Theorem}

\newtheorem{prop}[thm]{Proposition}

\theoremstyle{definition}

\newtheorem*{case*}{Case}

\newtheorem*{defn*}{Definition}

\newtheorem*{exmp*}{Example}

\renewcommand{\thestep}{}

\theoremstyle{remark}

\renewcommand{\thecase}{}

\newtheorem{rmk}[thm]{Remark}
\newtheorem*{rmk*}{Remark}

\makeatletter
\def\alphenumi{
  \def\theenumi{\alph{enumi}}
  \def\p@enumi{\theenumi}
  \def\labelenumi{(\@alph\c@enumi)}}
\makeatother

\makeatletter
\def\thecase{\@arabic\c@case}
\makeatother

\makeatletter
\def\thestep{\@arabic\c@step}
\makeatother

\newcount\hh
\newcount\mm
\mm=\time
\hh=\time
\divide\hh by 60
\divide\mm by 60
\multiply\mm by 60
\mm=-\mm
\advance\mm by \time
\def\hhmm{\number\hh:\ifnum\mm<10{}0\fi\number\mm}

\let\oldmarginpar\marginpar
\renewcommand\marginpar[1]{\-\oldmarginpar[\raggedleft\footnotesize #1]%
{\raggedright\footnotesize #1}}

\newcommand\NN{\mathbb{N}}

\newcommand\RR{\mathbb{R}}

\newcommand\TT{\mathbb{T}}

\newcommand\ZZ{\mathbb{Z}}

\newcommand\fg{{\mathfrak{g}}}

\newcommand\sE{{\mathscr{E}}}

\newcommand\sI{{\mathscr{I}}}

\newcommand\sU{{\mathscr{U}}}
\newcommand\sV{{\mathscr{V}}}
\newcommand\sW{{\mathscr{W}}}

\newcommand\bU{{\mathbf{U}}}

\newcommand\eps{\varepsilon}

\newcommand\GL{\operatorname{GL}}

\newcommand\SO{\operatorname{SO}}

\newcommand\SU{\operatorname{SU}}

\newcommand\ad{{\operatorname{ad}}}
\newcommand\Ad{{\operatorname{Ad}}}

\newcommand\Aut{\operatorname{Aut}}

\newcommand\dist{\operatorname{dist}}

\newcommand\End{\operatorname{End}}

\newcommand{\esssup}{\operatornamewithlimits{ess\ sup}}

\DeclareMathOperator{\genus}{genus}

\DeclareMathOperator{\Inj}{Inj}

\newcommand\Ker{\operatorname{Ker}}

\newcommand\Ric{\operatorname{Ric}}

\newcommand\Riem{\operatorname{Riem}}

\newcommand\vol{\operatorname{vol}}
\newcommand\Vol{\operatorname{Vol}}

\newcommand\apriori{{\emph{a priori }}}
\newcommand\Apriori{{\emph{A priori }}}

\newcommand\id{{\mathrm{id}}}

\newcommand\mutatis{{\emph{mutatis mutandis }}}

\numberwithin{equation}{section}

\usepackage{amssymb}
\usepackage{graphicx}
\usepackage{hyperref}
\usepackage{mathrsfs}
\usepackage[usenames]{color}
\usepackage{paralist}
\usepackage{slashed}
\usepackage{url}
\usepackage{verbatim}
\usepackage{xr}

\hypersetup{pdftitle={Energy gap for Yang--Mills connections, II: Arbitrary closed Riemannian manifolds}}
\hypersetup{pdfauthor={Paul M. N. Feehan}}

\begin{document}

\title[Energy gap for Yang--Mills connections, II]{Energy gap for Yang--Mills connections, II: Arbitrary closed Riemannian manifolds}

\author[Paul M. N. Feehan]{Paul M. N. Feehan}
\address{Department of Mathematics, Rutgers, The State University of New Jersey, 110 Frelinghuysen Road, Piscataway, NJ 08854-8019, United States}
\email{feehan@math.rutgers.edu}

\date{July 15, 2019. This version incorporates a correction to a significant error noticed since the publication (May 25, 2017) as \cite{Feehan_yangmillsenergygapflat_aim} and the circulation of our previous arXiv version \cite{Feehan_yangmillsenergygapflat_arxiv_v7} (updated July 10, 2017). The main correction in this version involves replacement of the role of \cite[Corollary 4.3]{UhlChern} due to Uhlenbeck and our \cite[Theorem 5.1]{Feehan_yangmillsenergygapflat_aim} by that of \cite[Theorems 1 and 9]{Feehan_nonlinear_uhlenbeck_estimate} due to the author. We provide a fuller discussion of the changes in our preprint, \emph{Corrigendum to ``Energy gap for Yang--Mills connections, II: Arbitrary closed Riemannian manifolds''}, July 15, 2019.}

\begin{abstract}
We prove an $L^{d/2}$ energy gap result for Yang--Mills connections on principal $G$-bundles, $P$, over arbitrary, closed, Riemannian, smooth manifolds of dimension $d\geq 2$. We apply our version of the {\L}ojasiewicz--Simon gradient inequality \cite{Feehan_yang_mills_gradient_flow_v4}, \cite{Feehan_Maridakis_Lojasiewicz-Simon_coupled_Yang-Mills_v4} to remove a positivity constraint on a combination of the Ricci and Riemannian curvatures in a previous $L^{d/2}$-energy gap result due to Gerhardt \cite[Theorem 1.2]{Gerhardt_2010} and a previous $L^\infty$-energy gap result due to Bourguignon, Lawson, and Simons \cite[Theorem C]{Bourguignon_Lawson_1981},
\cite[Theorem 5.3]{Bourguignon_Lawson_Simons_1979}, as well as an $L^2$-energy gap result due to Nakajima \cite[Corollary 1.2]{Nakajima_1987} for a Yang--Mills connection over the sphere, $S^d$, but with an arbitrary Riemannian metric.
\end{abstract}


\subjclass[2010]{Primary 58E15, 57R57; secondary 37D15, 58D27, 70S15, 81T13}

\keywords{Energy gaps, flat connections, flat bundles, gauge theory, {\L}ojasiewicz--Simon gradient inequality, Morse theory on Banach manifolds, closed Riemannian manifolds, Yang--Mills connections}

\thanks{Paul Feehan was partially supported by National Science Foundation grant DMS-1510064, the Oswald Veblen Fund and Fund for Mathematics (Institute for Advanced Study, Princeton), and the Max Planck Institute for Mathematics, Bonn.}

\maketitle
\tableofcontents

\section{Introduction}
\label{sec:Introduction}

\subsection{Main result}
\label{subsec:Main_result}
The purpose of our article to establish the following

\begin{mainthm}[$L^{d/2}$-energy gap for Yang--Mills connections]
\label{mainthm:Main}
Let $G$ be a compact Lie group and $P$ be a principal $G$-bundle over a closed, smooth manifold, $X$, of dimension $d \geq 2$ and endowed with a smooth Riemannian metric, $g$. Then there is a positive constant, $\eps = \eps(g, G) \in (0, 1]$, with the following significance.
If $A$ is a smooth Yang--Mills connection on $P$ with respect to the metric, $g$, and its curvature, $F_A$, obeys
\begin{equation}
\label{eq:Curvature_Ldover2_small}
\|F_A\|_{L^{d/2}(X)} \leq \eps,
\end{equation}
then $A$ is a flat connection.
\end{mainthm}

The notation in Theorem \ref{mainthm:Main} and throughout our Introduction is standard \cite{DK, FU, FrM}, but explained in Section \ref{sec:Preliminaries}. The quantity appearing in \eqref{eq:Curvature_Ldover2_small},
$\|F_A\|_{L^{d/2}(X)}$, depends on the Riemannian metric, $g$, only through its conformal equivalence class.

Previous energy gap results for Yang--Mills connections \cite{Bourguignon_Lawson_1981, Bourguignon_Lawson_Simons_1979, Dodziuk_Min-Oo_1982, DK, Gerhardt_2010, Min-Oo_1982, ParkerGauge} all required some positivity hypothesis on the curvature tensor, $\Riem_g$, of a Riemannian metric, $g$, on the manifold, $X$. Nakajima has established an $L^2$-energy gap result \cite[Corollary 1.2]{Nakajima_1987} for a Yang--Mills connection over $S^d$, but with an arbitrary Riemannian metric. His method employs a pointwise local decay estimate for a Yang--Mills connection established with the aid of a version of the monotonicity formula due to Price \cite{Price_1983}, extending earlier pointwise local decay estimates due to Uhlenbeck \cite{UhlRem} in dimension four.

The intuition underlying our proof of Theorem \ref{mainthm:Main} is rather that an energy gap must exist because otherwise one could have non-minimal Yang--Mills connections with $L^{d/2}$-energy arbitrarily close to zero and this should violate the analyticity of the Yang--Mills $L^2$-energy function, as manifested in the {\L}ojasiewicz--Simon gradient inequality established by the author for $d=2,3,4$ in \cite[Theorem 23.17]{Feehan_yang_mills_gradient_flow_v4}, by the author and Maridakis in \cite[Theorem 3]{Feehan_Maridakis_Lojasiewicz-Simon_coupled_Yang-Mills_v4} for arbitrary $d \geq 2$, and by R\r{a}de in \cite[Proposition 7.2]{Rade_1992} when $d=2,3$. The two other crucial ingredients in the proof of Theorem \ref{mainthm:Main} are Theorem \ref{thm:Uhlenbeck_3-5}, due to Uhlenbeck \cite[Theorem 3.5]{UhlRem}, and Theorem \ref{thm:Uhlenbeck_Chern_corollary_4-3}, due to the author \cite[Theorem 9]{Feehan_nonlinear_uhlenbeck_estimate} and correcting an earlier result due to Uhlenbeck \cite[Corollary 4.3]{UhlChern}, which we had attempted to reprove as \cite[Theorem 5.1]{Feehan_yangmillsenergygapflat_aim}. We refer the reader to Section \ref{subsec:Outline} for an outline of the proof of Theorem \ref{mainthm:Main}.

The existence of non-minimal Yang--Mills connections when $d=4$ was proved by Sibner, Sibner, and Uhlenbeck \cite{SibnerSibnerUhlenbeck} for the case of $X=S^4$ with its standard round metric of radius one, $G = \SU(2)$, and $P = S^4\times\SU(2)$.

In the setting of four-dimensional manifolds, the author \cite{Feehan_yangmillsenergygap} established $L^2$-energy gap results for Yang--Mills connections that also do \emph{not} require any positivity hypothesis on $\Riem_g$. Our previous results \cite[Theorem 1 and Corollary 2]{Feehan_yangmillsenergygap} replace the condition \eqref{eq:Curvature_Ldover2_small} by
\[
\|F_A^+\|_{L^2(X)} \leq \eps,
\]
and conclude that $F_A^+ \equiv 0$ on $X$ and $A$ is necessarily anti-self-dual with respect to the metric $g$ (and thus an absolute minimum of the Yang--Mills $L^2$-energy function). By reversing orientations on $X$, one obtains the analogous conclusion that $F_A^- \equiv 0$ on $X$ and $A$ is necessarily self-dual when $F_A^-$ is $L^2$-small. However, \cite[Corollary 2]{Feehan_yangmillsenergygap} does require that $g$ is generic in the sense of \cite{DK, FU} and that $G$, $P$ and $X$ obey at least one of three combinations of mild conditions involving the topology of $P$ and $X$, the representation variety of $\pi_1(X)$ in $G$, the choice of $G$, and the non-existence of flat connections on $P$. Our Theorem \ref{mainthm:Main} extends the main results of our companion article \cite{Feehan_yangmillsenergygap} to the case of arbitrary dimensions $d \geq 2$.

\subsection{Comparison with previous Yang--Mills energy gap results}
\label{subsec:Comparison_previous_Yang-Mills_energy_gap_results}
It is natural to separately consider the case of manifolds of arbitrary dimension $d \geq 2$ and the case $d =4$.

\subsubsection{Riemannian manifolds of arbitrary dimension $d \geq 2$}
\label{subsubsec:Riemannian_manifolds_arbitrary_dimension}
In \cite[Theorem 5.3]{Bourguignon_Lawson_Simons_1979}, Bourguignon, Lawson, and Simons asserted that if $d \geq 3$ and $X$ is the $d$-dimensional sphere, $S^d$, with its standard round metric of radius one, and $A$ is a Yang--Mills connection on a principal $G$-bundle, $P$, over $S^d$ such that
\begin{equation}
\label{eq:Linfinity_energy_gap_sphere}
\|F_A\|_{L^\infty(S^d)}^2 < \frac{1}{2}\binom{d}{2},
\end{equation}
then $A$ is flat. A detailed proof of this gap result is provided by Bourguignon and Lawson in
\cite[Theorem 5.19]{Bourguignon_Lawson_1981} for $d \geq 5$, \cite[Theorem 5.20]{Bourguignon_Lawson_1981} for $d=4$ (by combining the cases of $L^\infty$-small $F_A^+$ and $F_A^-$), and
\cite[Theorem 5.25]{Bourguignon_Lawson_1981} for $d=3$. (The results for the cases $d \geq 5$, $d=4$, and $d=3$ are combined in their \cite[Theorem C]{Bourguignon_Lawson_1981}.) These gap results are proved with the aid of the \emph{Bochner--Weitzenb\"ock formula}
\cite[Theorem 3.10 and Equation (5.1)]{Bourguignon_Lawson_1981} (compare \cite[Corollary II.3]{Lawson}),
\begin{equation}
\label{eq:Bourguignon_Lawson_1981_Theorem 3-10_and_Equation_5-1_Hodge_Laplacian}
\Delta_Av = \nabla_A^*\nabla_Av + v \circ (\Ric_g \wedge I + 2\Riem_g)
+ \{F_A, v\}, \quad\forall\, v \in \Omega^2(X; \ad P),
\end{equation}
for the \emph{Hodge Laplacian},
\begin{equation}
\label{eq:Lawson_page_93_Hodge_Laplacian}
\Delta_A := d_A^*d_A + d_Ad_A^* \quad\hbox{on } \Omega^2(X;\ad P),
\end{equation}
where $\Riem_g$ is the Riemann curvature tensor, $\Ric_g$ is the Ricci curvature tensor defined by $g$ and $\{F_A, \cdot\}: \Omega^2(X; \ad P) \to \Omega^2(X; \ad P)$ is defined in
\cite[Equation (3.7)]{Bourguignon_Lawson_1981}, so $\{F_A,v\}$ is a bilinear, pointwise, universal combination of $F_A$ and $v \in \Omega^2(X; \ad P)$, the operation $\circ$ is defined in
\cite[Equation (3.8)]{Bourguignon_Lawson_1981}, and $\Ric_g \wedge I$  is defined in
\cite[Equation (3.9)]{Bourguignon_Lawson_1981}. Thus,
\begin{multline*}
(\Ric_g \wedge I + 2\Riem_g)_{\xi_1,\xi_2}\xi_3
=
\Ric_g(\xi_1)\xi_3 \wedge \xi_2 + \xi_1\wedge \Ric_g(\xi_2)\xi_3 + 2\Riem_g(\xi_1,\xi_2)\xi_3
\in
C^\infty(TX),
\\
\quad\forall\, \xi_1, \xi_2, \xi_3 \in C^\infty(TX).
\end{multline*}
When $X=S^d$ with its standard round metric of radius one, then \cite[Corollary 3.14]{Bourguignon_Lawson_1981}
\[
(\Ric_g \wedge I + 2\Riem_g)_{\xi_1,\xi_2} = 2(d-2)\xi_1\wedge \xi_2,
\quad \forall\, \xi_1, \xi_2 \in C^\infty(TS^d).
\]
In the penultimate paragraph prior to the statement of their
\cite[Theorem 5.26]{Bourguignon_Lawson_1981}, Bourguignon and Lawson imply that these gap results continue to hold for a closed manifold, $X$, if the operator $\Ric_g \wedge I + 2\Riem_g$ has a positive least eigenvalue,
\begin{equation}
\label{eq:Bourguignon_Lawson_1981_equation_5-1_BW_Riemann_curvature_positivity}
\Ric_g \wedge I + 2\Riem_g \geq \lambda_g > 0,
\end{equation}
and the condition \eqref{eq:Linfinity_energy_gap_sphere} is generalized to
\begin{equation}
\label{eq:Linfinity_energy_gap_manifold}
\|F_A\|_{L^\infty(X)}^2 < \frac{1}{16}\frac{d(d-1)}{(d-2)^2}\lambda_g^2.
\end{equation}
This observation of Bourguignon and Lawson was improved by Gerhardt as \cite[Theorem 1.2]{Gerhardt_2010} by replacing the $L^\infty$ condition \eqref{eq:Linfinity_energy_gap_manifold} with
\begin{equation}
\label{eq:Ldover2_energy_gap_Gerhardt}
\|F_A\|_{L^{d/2}(X)} < \eps_0,
\end{equation}
for a positive constant, $\eps_0$, depending at most on $\lambda_g$, the Sobolev constant of $(X,g)$ for the embedding $W^{1,2}(X) \subset L^{2d/(d-2)}(X)$ (from \cite[Equation (2.26)]{Gerhardt_2010}), $d$, and the dimension of the Lie group, $G$. This result was also extended by him to the case where $(X,g)$ is a complete, non-compact manifold \cite[Theorem 1.3]{Gerhardt_2010}.

The positivity condition \eqref{eq:Bourguignon_Lawson_1981_equation_5-1_BW_Riemann_curvature_positivity} is assured if the (self-adjoint) curvature operator \cite[Section 1]{Hamilton_1986}, \cite[Section 3.1.2]{Petersen_riemannian_geometry3},
\begin{equation}
\label{eq:Curvature_operator}
\Riem_g:\Lambda_x^2 \to \Lambda_x^2,
\end{equation}
defined by the Riemannian metric, $g$, is positive at each point $x \in X$ \cite[p. 74]{Bourguignon_Karcher_1978}. (Here, we denote $\Lambda_x^2 = \Lambda^2(T_xX)$.) For such a metric, it is known that $X$ must be a real homology sphere by a theorem of Gallot and Meyer
\cite[Theorem A.5]{Chavel}, \cite{GallotMeyer}. Hence, the manifolds where one can apply the energy gap results of Bourguignon, Lawson, and Simons \cite{Bourguignon_Lawson_1981, Bourguignon_Lawson_Simons_1979} and Gerhardt \cite{Gerhardt_2010} have very strong constraints on their topology.

For a principal $G$-bundle over a closed, smooth manifold, $X$, with an arbitrary Riemannian metric, $g$, T. Huang \cite{Huang_2015arxiv} has proved that if $P$ admits a Yang--Mills connection $A$ whose curvature obeys \eqref{eq:Ldover2_energy_gap_Gerhardt}, then $P$ admits some flat connection, $\Gamma$. While T. Huang also published a proof of Theorem \ref{mainthm:Main} in \cite{Huang_2017}, his argument relied on the estimate in \cite[Corollary 4.3]{UhlChern} due to Uhlenbeck and which is incorrect as we explain in \cite[Appendix A]{Feehan_nonlinear_uhlenbeck_estimate}.

\subsubsection{Four-dimensional Riemannian manifolds}
\label{subsubsec:Four-dimensional_Riemannian_manifolds}
When $X$ is the four-dimensional sphere, $S^4$, with its standard round metric of radius one, the energy gap result of Bourguignon and Lawson \cite[Theorem C]{Bourguignon_Lawson_1981} was improved by Donaldson and Kronheimer \cite[Lemma 2.3.24]{DK} by relaxing the $L^\infty$ condition \eqref{eq:Linfinity_energy_gap_sphere} to the $L^2$ condition \eqref{eq:Curvature_Ldover2_small} (with $d=4$).

When $d=4$, more refined gap results have been established, based on the splitting
\cite[Sections 1.1.5, 1.1.6, and 2.1.3]{DK}
of two-forms into anti-self-dual and self-dual two forms, $\Omega^2(X) = \Omega^+(X)\oplus\Omega^-(X)$, and
\[
\Omega^2(X; \ad P) = \Omega^+(X; \ad P)\oplus \Omega^-(X; \ad P),
\]
and the resulting Bochner--Weitzenb\"ock formulae for the restrictions of the Hodge Laplacian,
\[
d_Ad_A^* + d_A^*d_A = 2d_A^\pm d_A^{\pm,*} \quad\text{on } \Omega^\pm(X; \ad P),
\]
namely \cite[Equation (6.26) and Appendix C, p. 174]{FU}, \cite[Equation (5.2)]{GroisserParkerSphere},
\begin{equation}
\label{eq:Freed_Uhlenbeck_6-26}
2d_A^+d_A^{+,*}v = \nabla_A^*\nabla_Av + \left(\frac{1}{3}R_g - 2w_g^+\right)v + \{F_A^+, v\},
\quad\forall\, v \in \Omega^+(X; \ad P),
\end{equation}
with the analogous formula for $2d_A^-d_A^{-,*}v$ when $v \in \Omega^-(X; \ad P)$.

In \cite[Theorem 5.4]{Bourguignon_Lawson_Simons_1979}, Bourguignon, Lawson, and Simons asserted that if $X$ is the sphere, $S^4$, with its standard round metric of radius one, and $A$ is a Yang--Mills connection on a principal $G$-bundle, $P$, over $S^4$ such that
\begin{equation}
\label{eq:Linfinity_energy_gap_4sphere}
\|F_A^+\|_{L^\infty(S^4)}^2 < 3,
\end{equation}
then $F_A^+ \equiv 0$ on $S^4$ and $A$ is anti-self-dual. By reversing orientations on $S^4$, one obtains the analogous conclusion that $F_A^- \equiv 0$ on $S^4$ and $A$ is necessarily self-dual when $\|F_A^-\|_{L^\infty(S^4)}^2 < 3$. A detailed proof of this gap result is provided by Bourguignon and Lawson \cite[Theorem 5.20]{Bourguignon_Lawson_1981}. (The result is also quoted as \cite[Theorem D]{Bourguignon_Lawson_1981}.)

More generally, for a smooth Riemannian metric, $g$, on a four-dimensional, oriented manifold, $X$, let $R_g(x)$ denote its scalar curvature at a point $x \in X$ and let $\sW_g^\pm(x) \in \End(\Lambda_x^\pm)$ denote its self-dual and anti-self-dual Weyl curvature tensors at $x$, where $\Lambda_x^2 = \Lambda_x^+\oplus \Lambda_x^-$. Define
$$
w_g^\pm(x) := \text{Largest eigenvalue of } \sW_g^\pm(x), \quad\forall\, x \in X.
$$
Bourguignon and Lawson prove \cite[Theorem 5.26]{Bourguignon_Lawson_1981} that if $X$ is a closed, four-dimensional, oriented, smooth manifold endowed with a Riemannian metric, $g$, with vanishing self-dual Weyl curvature ($\sW_g^+ \equiv 0$ on $X$), positive scalar curvature, $R_g > 0$ on $X$, and $A$ is a Yang--Mills connection on a principal $G$-bundle, $P$, over $X$ whose curvature, $F_A$, obeys the pointwise bound,
\begin{equation}
\label{eq:Linfinity_energy_gap_4manifold}
|F_A^+| < \frac{R_g}{4} \quad\text{on } X,
\end{equation}
then $F_A^+ \equiv 0$ on $X$ and $A$ is anti-self-dual with respect to the metric, $g$. By reversing orientations on $X$, one obtains the analogous conclusion that $F_A^- \equiv 0$ on $X$ and $A$ is necessarily self-dual with respect to the metric, $g$, when $|F_A^-| < R_g/4$ on $X$.

The result \cite[Theorem 5.26]{Bourguignon_Lawson_1981} due to Bourguignon and Lawson was extended by Min-Oo \cite[Theorem 2]{Min-Oo_1982} and Parker \cite[Proposition 2.2]{ParkerGauge}, in the sense that the pointwise condition \eqref{eq:Linfinity_energy_gap_4manifold} and assumption that $\sW_g^+ \equiv 0$ on $X$ were relaxed to the $L^2$-energy condition,
\begin{equation}
\label{eq:Minoo_theorem_2}
\|F_A^+\|_{L^2(X)} \leq \eps,
\end{equation}
for a closed manifold, $X$, for which $\Riem_g$ obeys the positivity condition,
\begin{equation}
\label{eq:Freed_Uhlenbeck_page_174_positive_metric_self-dual}
\frac{1}{3}R_g - 2w_g^+ > 0 \quad\hbox{on } X.
\end{equation}
As usual, the analogous conclusion for that $F_A^{-,g} \equiv 0$ on $X$ when $A$ is a Yang--Mills connection with $L^2$-small enough $F_A^{-,g}$ and $\Riem_g$ obeys the positivity condition,
\begin{equation}
\label{eq:Freed_Uhlenbeck_page_174_positive_metric_anti-self-dual}
\frac{1}{3}R_g - 2w_g^- > 0 \quad\hbox{on } X,
\end{equation}
follows by reversing orientations on $X$.

Unfortunately, the hypothesis \eqref{eq:Freed_Uhlenbeck_page_174_positive_metric_self-dual} also imposes strong constraints on the topology of $X$, as the Bochner--Weitzenb\"ock formula \eqref{eq:Freed_Uhlenbeck_6-26} implies that the dimension of the vector space of harmonic, real, self-dual two-forms is zero. Hence, $b^+(X) = 0$ and the bilinear intersection form, $Q$ on the cohomology group, $H^2(X;\ZZ)$, is negative definite \cite[Section 1.1.6]{DK}.

As we already described, we have extended the result \cite[Theorem 2]{Min-Oo_1982} of Min-Oo and \cite[Proposition 2.2]{ParkerGauge} of Parker in our \cite[Theorem 1 and Corollary 2]{Feehan_yangmillsenergygap} by removing the positivity conditions \eqref{eq:Freed_Uhlenbeck_page_174_positive_metric_self-dual} and \eqref{eq:Freed_Uhlenbeck_page_174_positive_metric_anti-self-dual}.

Extensions of \cite[Theorem 2]{Min-Oo_1982}, \cite[Proposition 2.2]{ParkerGauge} to the case where $(X,g)$ is a complete, non-compact, oriented Riemannian, smooth manifold have been obtained by Dodziuk and Min-Oo \cite{Dodziuk_Min-Oo_1982}, Shen \cite{Shen_1982}, and Xin \cite{Xin_1984}.

\subsection{Further research}
\label{subsec:Further_research}
We discuss possible extensions of Theorem \ref{mainthm:Main} and potential applications of our method of proof to other problems in geometric analysis and mathematical physics.

\subsubsection{Complete non-compact Riemannian manifolds}
\label{subsubsec:Further_research_complete}
It is possible that Theorem \ref{mainthm:Main} might extend to the setting of complete, non-compact Riemannian manifolds (that do not admit conformal compactifications), thus generalizing the previous results in this setting due to Dodziuk and Min-Oo \cite{Dodziuk_Min-Oo_1982}, Gerhardt \cite{Gerhardt_2010}, Shen \cite{Shen_1982}, and Xin \cite{Xin_1984}. However, it is likely that any such extensions would be fairly technical in nature. One obstacle lies in the required generalization of the {\L}ojasiewicz--Simon gradient inequality from the setting of compact to complete manifolds and that would probably entail restrictions on the allowable ends of $X$, such as the cylindrical ends employed by Morgan, Mrowka, and Ruberman \cite{MMR} and Taubes \cite{TauL2}, together with their use of weighted Sobolev spaces adapted to such cylindrical ends \cite{Lockhart_McOwen_1985}.

\subsubsection{Adaptation of the gradient inequality paradigm to other problems in geometric analysis and mathematical physics}
\label{subsubsec:Further_research_geometric_analysis_and_math_physics}
Energy gap or quantization results are not confined to the realm of Yang--Mills gauge theory, as evidenced by recent results of Bernard and Rivi{\'e}re \cite{Bernard_Riviere_2014} on Willmore surfaces (critical points of the Willmore energy function), older results on harmonic maps, such those of Xin \cite{Xin_1980}, and elsewhere. The {\L}ojasiewicz--Simon gradient inequality, originally due to Simon \cite{Simon_1983}, has now been established in great generality (see S.-Z. Huang \cite{Huang_2006} or our monograph \cite{Feehan_yang_mills_gradient_flow_v4} for surveys and references), so it is reasonable to expect that the methods of our article may extend beyond their present context in Yang--Mills gauge theory, particularly in situations where previous results have relied on Bochner--Weitzenb\"ock formulae and positive curvature hypotheses. While analyticity of the energy function is required by the {\L}ojasiewicz--Simon gradient inequality, there are other gradient inequalities which do \emph{not} require analyticity \cite{Huang_2006}.

\subsection{Outline}
\label{subsec:Outline}
In Section \ref{sec:Preliminaries}, we establish our notation and recall basic definitions in gauge theory over Riemannian manifolds required for the remainder of this article. Section \ref{sec:Flat_connections} reviews essential background material concerning flat connections on a principal $G$-bundles, including Uhlenbeck compactness of the moduli space of flat connections in Section \ref{subsec:Uhlenbeck_compactness_moduli_space_flat_connections} and a special case (Corollary \ref{cor:Rade_proposition_7-2_flat}) of our {\L}ojasiewicz--Simon gradient inequality for the Yang--Mills $L^2$-energy function \cite[Theorem 23.17]{Feehan_yang_mills_gradient_flow_v4}, \cite[Theorem 3]{Feehan_Maridakis_Lojasiewicz-Simon_coupled_Yang-Mills_v4}. In Section \ref{sec:Uhlenbeck_Ldover2_small_curvature_and_estimate_curvature_Yang-Mills_connection}, we
recall results due to Uhlenbeck concerning an \apriori estimate for the curvature of a Yang--Mills connection \cite{UhlRem} and the existence of a local Coulomb gauge for a connection, $A$, with $L^{d/2}$-small curvature, $F_A$ \cite{UhlLp}. Section \ref{sec:Uhlenbeck_approach_existence_flat_connection_and_apriori_estimate} contains the statement of our Theorem \ref{thm:Uhlenbeck_Chern_corollary_4-3}, correcting the statement of \cite[Corollary 4.3]{UhlChern} due to Uhlenbeck and our previous \cite[Theorem 5.1]{Feehan_yangmillsenergygapflat_aim}. Theorem \ref{thm:Uhlenbeck_Chern_corollary_4-3} provides existence of a flat connection, $\Gamma$, on $P$ given a Sobolev connection on $P$ with $L^p$-small curvature (when $p > (\dim X)/2$), a global gauge transformation, $u$, of $A$ to Coulomb gauge with respect to $\Gamma$, and a Sobolev norm estimate for the distance between $A$ and $\Gamma$. We complete the proof of Theorem \ref{mainthm:Main} in Section \ref{sec:Completion_proof_main_theorem}.

Appendix \ref{app:Alternative_proofs_simplifying_hypotheses}
contains proofs (or summaries) of several results described in this article that simplify considerably under the assumption of additional hypotheses, including Theorem \ref{mainthm:Main} in Section \ref{subsec:BW_formula_existence_Ldover2_energy_gap} (under a certain positive curvature hypothesis); the estimates \eqref{eq:Uhlenbeck_Chern_corollary_4-3_uA-Gamma_W1p_bound} of the Sobolev distance between $A$ and a flat connection, $\Gamma$, in Section \ref{subsec:Estimate_Sobolev_distance_flat_connection_Hodge_Laplacian_vanishing_kernel} (under the hypothesis that the Hodge Laplacian for $\Gamma$ on $\Omega^2(X;\ad P)$ has vanishing kernel); and the first part of Theorem \ref{thm:Uhlenbeck_Chern_corollary_4-3} in Section \ref{subsec:Yang_existence_flat_connection} (existence of a flat connection under the hypothesis that $P$ supports a smooth connection with $L^\infty$-small curvature).

Appendix \ref{sec:Corrections} contains a list of the mathematical corrections to the previously published version \cite{Feehan_yangmillsenergygapflat_aim} of our article that are provided in this version. The most significant correction involves replacement of the role of \cite[Corollary 4.3]{UhlChern} due to Uhlenbeck and our previous \cite[Theorem 5.1]{Feehan_yangmillsenergygapflat_aim} by that of \cite[Theorems 1 and 9]{Feehan_nonlinear_uhlenbeck_estimate} due to the author. We provide a fuller discussion of these corrections in \cite{Feehan_yangmillsenergygapflat_corrigendum}.

\subsection{Acknowledgments}
\label{subsec:Acknowledgments}
I am very grateful to the Max Planck Institute for Mathematics, Bonn, for their hospitality during the preparation of this article. I would also like to thank Blaine Lawson for kind comments regarding this and its companion article \cite{Feehan_yangmillsenergygap}, Hiraku Nakajima for helpful communications, and Karen Uhlenbeck for helpful comments regarding her article \cite{UhlChern}, Changyou Wang for alerting me to subtleties particular to dimension two, and Baozhong Yang for helpful comments regarding his article \cite{Yang_2003pjm}. This correction to the previously published version \cite{Feehan_yangmillsenergygapflat_aim} of our article owes its existence to Tom Mrowka and his description of a counterexample \cite{Mrowka_7-30-2018} to exponential convergence (implied by our \cite[Theorem 2 and Remark 1.10]{Feehan_lojasiewicz_inequality_ground_state_v1}) for Yang--Mills gradient flow near the moduli space of flat connections. I am indebted to him and very grateful to Mariano Echeverria, Kenji Fukaya, Takeo Nishinou, Tom Parker, Nikolai Saveliev, Penny Smith, Katrin Wehrheim, and Graeme Wilkin for helpful communications during the preparation of this correction to \cite{Feehan_yangmillsenergygapflat_aim}. 

\section{Preliminaries}
\label{sec:Preliminaries}
We shall generally adhere to the now standard gauge-theory conventions and notation of Donaldson and Kronheimer \cite{DK}, Freed and Uhlenbeck \cite{FU}, and Friedman and Morgan \cite{FrM}; those references and our monograph \cite{Feehan_yang_mills_gradient_flow_v4} also provide the necessary background for our article.

Throughout our article, $G$ denotes a compact Lie group and $P$ a smooth principal $G$-bundle over a closed, smooth manifold, $X$, of dimension $d \geq 2$ and endowed with Riemannian metric, $g$. We denote $\Lambda^l := \Lambda^l(T^*X)$ for integers $l\geq 1$ and $\Lambda^0 = X\times\RR$, and let\footnote{We follow the notational conventions of Friedman and Morgan \cite[p. 230]{FrM}, where they define $\ad P$ as we do here and define $\Ad P$ to be the group of automorphisms of the principal $G$-bundle, $P$.}
$\ad P := P\times_{\ad}\fg$ denote the real vector bundle associated to $P$ by the adjoint representation of $G$ on its Lie algebra,
$\Ad:G \ni u \to \Ad_u \in \Aut\fg$. We fix a $G$-invariant inner product on the Lie algebra $\fg$ and thus define a fiber metric on $\ad P$. (When $G$ is semi-simple, one may use the Killing form to define a $G$-invariant inner product $\fg$.) Given a $C^\infty$ reference connection, $A_1$, on $P$, we let
\begin{align*}
\nabla_{A_1}: C^\infty(X;\Lambda^l\otimes\ad P) &\to C^\infty(X; T^*X\otimes \Lambda^l\otimes\ad P),
\\
d_{A_1}: C^\infty(X; \Lambda^l\otimes\ad P) &\to C^\infty(X; \Lambda^{l+1}\otimes\ad P), \quad l \in \NN,
\end{align*}
denote the \emph{covariant derivative} \cite[Equation (2.1.1)]{DK} and \emph{exterior covariant derivative} \cite[Equation (2.1.12)]{DK} associated with $A_1$, respectively. We write the set of non-negative integers as $\NN$ and abbreviate $\Omega^l(X; \ad P) :=  C^\infty(X; \Lambda^l\otimes\ad P)$, the Fr\'echet space of $C^\infty$ sections of $\Lambda^l\otimes\ad P$.

We denote the Banach space of sections of $\Lambda^l\otimes\ad P$ of Sobolev class $W^{k,q}$, for any $k\in \NN$ and $q \in [1,\infty]$, by $W_{A_1}^{k,q}(X; \Lambda^l\otimes\ad P)$, with norm,
\[
\|\phi\|_{W_{A_1}^{k,q}(X)} := \left(\sum_{j=0}^k \int_X |\nabla_{A_1}^j\phi|^q\,d\vol_g \right)^{1/q},
\]
when $1\leq q<\infty$ and
\[
\|\phi\|_{W_{A_1}^{k,\infty}(X)} := \sum_{j=0}^k \esssup_X |\nabla_{A_1}^j\phi|,
\]
otherwise, where $\phi \in W_{A_1}^{k,q}(X; \Lambda^l\otimes\ad P)$.

We define the \emph{Yang--Mills $L^2$-energy function} by
\begin{equation}
\label{eq:Yang-Mills_energy_functional}
\sE_g(A)  := \frac{1}{2}\int_X |F_A|^2\,d\vol_g,
\end{equation}
where $A$ is a connection on $P$ of Sobolev class $W^{k,q}$ and curvature \cite[Equation (2.1.13)]{DK},
\[
F_A = d_A\circ d_A \in W^{k-1,q}(X; \Lambda^2\otimes\ad P).
\]
To ensure that the integral \eqref{eq:Yang-Mills_energy_functional} is well-defined, we require that $k\geq 1$ and $q \geq 1$ obey
\begin{inparaenum}[(\itshape i\upshape)]
\item $q\geq 2$ if $k=1$, or
\item $q^* \equiv (k-1)q/(d-(k-1)q) \geq 2$ if $k\geq 2$ and $(k-1)q<d$,
\item $k\geq 2$ and $(k-1)q \geq d$,
\end{inparaenum}
so in each case $W^{k-1,q}(X;\RR) \subset L^2(X;\RR)$ by the Sobolev Embedding \cite[Theorem 4.12]{AdamsFournier}.

A connection, $A$ on $P$, is a \emph{critical point} of $\sE_g$ --- and by definition a \emph{Yang--Mills connection} with respect to the metric $g$ --- if and only if it obeys the \emph{Yang--Mills equation} with respect to the metric $g$,
\[
d_A^{*,g}F_A = 0 \quad\hbox{a.e. on } X,
\]
since $d_A^{*,g}F_A = \sE_g'(A)$ when the gradient of $\sE = \sE_g$ is defined by the $L^2$ metric \cite[Section 6.2.1]{DK} and $d_A^* = d_A^{*,g}: \Omega^l(X; \ad P) \to \Omega^{l-1}(X; \ad P)$ is the $L^2$ adjoint of  $d_A:\Omega^l(X; \ad P) \to \Omega^{l+1}(X; \ad P)$. By contrast, the curvature, $F_A$, of a connection always obeys the \emph{Bianchi identity} \cite[Equation (2.1.21)]{DK},
\[
d_AF_A = 0 \quad\hbox{a.e. on } X.
\]
In the sequel, constants are generally denoted by $C$ (or $C(*)$ to indicate explicit dependencies) and may increase from one line to the next in a series of inequalities. We write $\eps \in (0,1]$ to emphasize a positive constant that is understood to be small or $K \in [1,\infty)$ to emphasize a constant that is understood to be positive but finite. We let $\Inj(X,g)$ denote the injectivity radius of a smooth Riemannian manifold, $(X,g)$.

\section{Flat connections and the {\L}ojasiewicz--Simon gradient inequality}
\label{sec:Flat_connections}
In this section, we recall some background material concerning flat connections on principal $G$-bundles that will be useful in the sequel. Section \ref{subsec:Existence_flat_connections} reviews related existence and non-existence results for flat connections. Section \ref{subsec:Flat_bundles} recalls the well-known equivalent characterizations of flat connections. In Section \ref{subsec:Uhlenbeck_compactness_moduli_space_flat_connections}, we describe Uhlenbeck's Weak Compactness Theorem for connections with $L^p$ bounds on curvature (when $p > (\dim X)/2$) and the resulting
Uhlenbeck compactness of the moduli space of flat connections on a principal $G$-bundle, $P$, for a compact Lie group, $G$. Finally, in Section \ref{subsec:Lojasiewicz-Simon gradient inequality_neighborhood_flat_connection} we review our {\L}ojasiewicz--Simon gradient inequality (Theorem \ref{thm:Rade_proposition_7-2}) for the Yang--Mills $L^2$-energy function, previously established in our monograph \cite{Feehan_yang_mills_gradient_flow_v4}.

\subsection{Existence and non-existence of flat connections}
\label{subsec:Existence_flat_connections}
To set Theorem \ref{mainthm:Main} in context, it is interesting to consider previous work on the existence of flat connections on a principal $G$-bundle, $P$, or an associated vector bundle, $E$, over a closed, connected, oriented, smooth manifold, $X$, of dimension $d \geq 2$. If a real (or complex) vector bundle over $X$ admits a flat connection, then all its Pontrjagin (or Chern) classes with rational coefficients are zero
\cite[Appendix C, Corollary 2, p. 308]{MilnorStasheff}. Hence, the vanishing of characteristic classes with rational coefficients of an associated vector bundle, $E$ is a necessary, but not sufficient condition for existence of flat connections on $E$ or $P$.

Aside from trivial non-existence results implied by the Chern-Weil formula, the first non-existence result for flat connections is due to Milnor \cite{Milnor_1958cmh}, who considered the case of $d=2$ and $G=\GL^+(2,\RR)$, the group of $2\times 2$ real matrices with positive determinant. His \cite[Theorem 1]{Milnor_1958cmh} asserts that $P$ does not admit a flat connection if $\chi(X)\geq \genus(X)$, where $\chi(X)=2-2\genus(X)$ is the Euler characteristic of $X$. Consequently, a Riemann surface, $X$, with $\genus(X) \geq 2$ does not admit an affine connection with curvature zero \cite[Corollary, p. 215]{Milnor_1958cmh}. (See \cite[Section III.3]{Kobayashi_Nomizu_v1} for an introduction to affine connections.) Related non-existence results are due to Matsushima and Okamoto \cite{Matsushima_Okamoto_1979}, who showed that if $G$ is a real semisimple Lie group, then $G$ has no left-invariant, torsion-free flat affine connection, generalizing an earlier result of Milnor \cite{Milnor_1976}. The sphere, $S^d$, does not admit a torsion-free flat affine connection for $d \geq 2$ because the fundamental group of $S^d$ is not infinite \cite{Agaoka_1995, Auslander_Markus_1955}.

Suppose now that $E$ is a holomorphic vector bundle over a compact, connected Riemann surface, $\Sigma$. A result due to Weil says that $E$ admits a flat connection if and only if each direct summand of $E$ is of degree zero \cite{Weil_1938}, \cite[Theorem 10, p. 203]{Atiyah_1957}. This criterion for the existence of flat connections was extended by Azad and Biswas \cite{Azad_Biswas_2003} to holomorphic principal $G$-bundles, $P$, over $\Sigma$, where $G$ is a connected reductive linear algebraic group over $\Sigma$. More generally, Biswas and Subramanian \cite{Biswas_Subramanian_2004} give a criterion for the existence of a flat connection on a principal $G$-bundle, $P$, over a projective manifold, $Z$, when the structure group, $G$, is not reductive. For a survey of research on existence of flat connections on principal bundles or associated vector bundles over complex manifolds, we refer the reader to Azad and Biswas \cite{Azad_Biswas_2003} and Biswas and Subramanian \cite{Biswas_Subramanian_2004}.

\subsection{Flat bundles}
\label{subsec:Flat_bundles}
Returning to the setting of connections on a principal $G$-bundle, $P$, over a real manifold, $X$, we recall the equivalent characterizations of \emph{flat bundles} \cite[Section 1.2]{Kobayashi}, that is, bundles admitting a flat connection.

Let $G$ be a Lie group and $P$ be a smooth principal $G$-bundle over a smooth manifold, $X$. Let $\{U_\alpha\}$ be an open cover of $X$ with local trivializations, $\tau_\alpha: P \restriction U_\alpha \cong U_\alpha\times G$. Let $g_{\alpha\beta}: U_\alpha\cap U_\beta \to G$ be the family of transition functions defined by $\{U_\alpha, \tau_\alpha\}$.
A \emph{flat structure} in $P$ is given by $\{U_\alpha, \tau_\alpha\}$ such that the $g_{\alpha\beta}$ are all constant maps. A connection in $P$ is said to be \emph{flat} if its curvature vanishes identically.

\begin{prop}[Characterizations of flat principal bundles]
\label{prop:Kobayashi_1-2-6}
(See \cite[Proposition 1.2.6]{Kobayashi}.)
For a smooth principal $G$-bundle $P$ over a smooth manifold, $X$, the following conditions are equivalent:
\begin{enumerate}
  \item $P$ admits a flat structure,
  \item $P$ admits a flat connection,
  \item $P$ is defined by a representation $\pi_1(X) \to G$.
\end{enumerate}
\end{prop}

Given a flat structure on $P$, we may construct a flat connection, $\Gamma$, on $P$ using the zero local connection one-forms, $\gamma_\alpha \equiv 0$ on $U_\alpha$, for each $\alpha$ as in \cite[Equation $(1.2.1')$]{Kobayashi} and observing that the compatibility conditions \cite[Equation (1.1.16)]{Kobayashi},
\[
0 = \gamma_\beta
= g_{\alpha\beta}^{-1}\gamma_\alpha g_{\alpha\beta} + g_{\alpha\beta}^{-1} dg_{\alpha\beta}
= 0 \quad\text{on } U_\alpha\cap U_\beta,
\]
are automatically obeyed.

\subsection{Uhlenbeck compactness of the moduli space of flat connections}
\label{subsec:Uhlenbeck_compactness_moduli_space_flat_connections}
In \cite[Chapter 4]{MMR}, the moduli space of gauge-equivalence classes of flat connections on the product bundle, $Q = Y\times \SU(2)$, over a closed, oriented, Riemannian three-dimensional manifold, $Y$, is called the \emph{character variety} of $Y$. (Every principal $\SU(2)$-bundle over a three-dimensional manifold is topologically trivial.) We note \cite[Proposition 2.2.3]{DK} that the gauge-equivalence classes of flat $G$-connections over a connected manifold, $X$, are in one-to-one correspondence with the conjugacy classes of representations $\pi_1(X) \to G$.

We recall from \cite{UhlLp} why
\[
M(P) := \{\Gamma: F_\Gamma = 0\}/\Aut(P),
\]
the moduli space of gauge-equivalence classes, $[\Gamma]$, of flat connections, $\Gamma$, on $P$ is compact, when $G$ is a compact Lie group. Suppose that $\{\Gamma_n\}_{n\in\NN}$ is a sequence of flat connections of class $W^{1,q}$ on $P$, where $q>d/2$. According to \cite[Theorem 1.3]{UhlLp}, over each geodesic ball $B_\rho(x_\alpha) \subset X$ (say with $\rho \in (0, \Inj(X,g)/2]$), there is a sequence of local gauge transformations, $u_{\alpha,n} \in \Aut(P\restriction B_\rho(x_\alpha))$ of class $W^{2,q}$, such that
\[
u_{\alpha,n}(\gamma_n) = 0 \quad\text{a.e. on } B_\rho(x_\alpha),
\]
where $\gamma_n  := \Gamma_n - \Theta \in W^{1,q}(B_\rho(x_\alpha); \Lambda^1\otimes\fg)$ is the sequence of local connection one-forms defined by the product connection, $\Theta$, on $P\restriction B_\rho(x_\alpha) \cong B_\rho(x_\alpha) \times G$. We can now appeal to a patching result for sequences of local connection one-forms and local gauge transformations --- for example \cite[Corollary 4.4.8]{DK}, which applies to a manifold $X$ of arbitrary dimension. We can thus conclude that, after passing to a subsequence, there is a sequence of global gauge transformations, $u_n \in \Aut(P)$, of class $W^{2,q}$, such that
\[
u_n(\Gamma_n) \to \Gamma_\infty \quad\text{(strongly) in } W^{1,q}(X; \Lambda^1\otimes\ad P)
\quad\text{as } n \to \infty,
\]
for some flat connection, $\Gamma_\infty$, of class $W^{1,q}$ on $P$. This conclusion could also be deduced from \cite[Theorem 1.5]{UhlLp}, noting that we obtain strong rather than weak convergence here because the local convergence over each ball $B_\rho(x_\alpha)$ is trivially strong (since the local connection one-forms are identically zero with respect to suitable trivializations of $P\restriction B_\rho(x_\alpha)$). (More generally, the arguments of \cite{UhlLp, UhlRem} can be used to show that the space of gauge-equivalence classes of Yang--Mills connections on a principal $G$-bundle over a closed, $d$-dimensional, Riemannian manifold with a uniform $L^p$ bound on curvature is compact when $p>d/2$.)

By contrast, the moduli space of gauge-equivalence classes of anti-self-dual connections on a principal $G$-bundle over a closed, four-dimensional, oriented, Riemannian manifold \cite[Section 4.4]{DK} is not compact. One has local elliptic estimates for connection one-forms in terms of curvature and if one also had a uniform $L^p$ bound, with $p>2$, on the curvature of anti-self-dual connections, then one would obtain compactness, just as above. However, because one only has a uniform $L^2$ bound on the curvature of anti-self-dual connections and the $L^{d/2}$ norm on two-forms over a $d$-manifold is conformally invariant, the argument fails.

\subsection{{\L}ojasiewicz--Simon gradient inequality on a Sobolev neighborhood of a flat connection}
\label{subsec:Lojasiewicz-Simon gradient inequality_neighborhood_flat_connection}
Our {\L}ojasiewicz--Simon gradient inequality for the Yang--Mills $L^2$-energy function is one of the key technical ingredients underlying the proof of our Theorem \ref{mainthm:Main}. We recall the statement we shall need from \cite{Feehan_Maridakis_Lojasiewicz-Simon_coupled_Yang-Mills_v4}.

\begin{thm}[{\L}ojasiewicz--Simon gradient inequality for the Yang--Mills $L^2$-energy function]
\label{thm:Rade_proposition_7-2}
(See \cite[Theorem 3]{Feehan_Maridakis_Lojasiewicz-Simon_coupled_Yang-Mills_v4}.)
Let $(X,g)$ be a closed, Riemannian, smooth manifold of dimension $d$, and $G$ be a compact Lie group, and $A_1$ be a connection of class $C^\infty$, and $A_\infty$ a Yang--Mills connection of class $W^{1,q}$, with $q \in [2,\infty)$ obeying $q>d/2$, on a principal $G$-bundle, $P$, over $X$. If $d\geq 2$ and $p \in [2,\infty)$ obeys $d/2 \leq p \leq q$, then there are positive constants $c$, $\sigma$, and $\theta \in [1/2,1)$, depending on $A_1$, $A_\infty$, $g$, $G$, $p$, and $q$ with the following significance.  If $A$ is a connection of class $W^{1,q}$ on $P$ and
\begin{equation}
\label{eq:Rade_7-1_neighborhood}
\|A - A_\infty\|_{W^{1,p}_{A_1}(X)} < \sigma,
\end{equation}
then
\begin{equation}
\label{eq:Rade_7-1}
\|d_A^*F_A\|_{W^{-1,p}_{A_1}(X)} \geq c|\sE(A) - \sE(A_\infty)|^\theta,
\end{equation}
where $\sE(A)$ is given by \eqref{eq:Yang-Mills_energy_functional}.
\end{thm}

By virtue of the compactness of the moduli space, $M(P)$, of gauge-equivalence classes of flat connections on $P$, described in Section \ref{subsec:Uhlenbeck_compactness_moduli_space_flat_connections}, we can deduce the following corollary to Theorem \ref{thm:Rade_proposition_7-2}.

\begin{cor}[{\L}ojasiewicz--Simon gradient inequality for the Yang--Mills energy function near flat connections]
\label{cor:Rade_proposition_7-2_flat}
Let $X$ be a closed, smooth manifold of dimension $d$ and endowed with a Riemannian metric, $g$, and $G$ be a compact Lie group. Assume that $d\geq 2$ and $p \in [2,\infty)$ obeys $p \geq d/2$. Then there are positive constants $c$, $\sigma$, and $\theta \in [1/2,1)$, depending on $g$, $G$, and $p$ with the following significance. If $A$ is a connection of class $W^{1,q}$ on a principal $G$-bundle, $P$, over $X$, with $q\in [2,\infty)$ obeying $q>d/2$ and $q \geq p$, such that
\begin{equation}
\label{eq:Rade_7-1_neighborhood_flat}
\|A - \Gamma\|_{W^{1,p}_\Gamma(X)} < \sigma,
\end{equation}
for some flat connection, $\Gamma$, of class $W^{1,q}$ on $P$, then
\begin{equation}
\label{eq:Rade_7-1_flat}
\|d_A^*F_A\|_{W^{-1,p}_\Gamma(X)} \geq c|\sE(A)|^\theta.
\end{equation}
\end{cor}

In particular, if $A$ is a Yang--Mills connection, then \eqref{eq:Rade_7-1_flat} implies that $A$ is necessarily flat. Hence, the proof of our main result, Theorem \ref{mainthm:Main}, will be complete provided we can show that a $W^{1,p}$ connection, $A$, with $L^p$-small enough curvature is $W_\Gamma^{1,p}(X)$-close enough to some flat connection, $\Gamma$, on $P$, for $p$ as in Corollary \ref{cor:Rade_proposition_7-2_flat}. We shall discuss the statement and proof of the latter result in Section \ref{sec:Uhlenbeck_approach_existence_flat_connection_and_apriori_estimate}.

\section{Connections with $L^{d/2}$-small curvature and a priori estimates for Yang--Mills connections}
\label{sec:Uhlenbeck_Ldover2_small_curvature_and_estimate_curvature_Yang-Mills_connection}
In this section we review several key results due to Uhlenbeck concerning an \apriori estimate for the curvature of a Yang--Mills connection \cite{UhlRem} and existence of a local Coulomb gauge for a connection with $L^{d/2}$-small curvature \cite{UhlLp}.

\subsection{Connections with $L^{d/2}$-small curvature}
\label{subsec:Uhlenbeck_Ldover2_small_curvature}
We recall a statement of Uhlenbeck's Theorem \cite{UhlLp} on existence of local Coulomb gauges (with a clarification due to Wehrheim \cite{Wehrheim_2004}), together with two extensions proved by the author in \cite{Feehan_lojasiewicz_inequality_ground_state_v1}.  

\begin{thm}[Existence of a local Coulomb gauge and \apriori estimate for a Sobolev connection with $L^{d/2}$-small curvature]
\label{thm:Uhlenbeck_Lp_1-3}
(Correction to our quotation \cite[Theorem 4.1]{Feehan_yangmillsenergygapflat_aim} of Uhlenbeck's \cite[Theorem 1.3 or Theorem 2.1 and Corollary 2.2]{UhlLp}; compare Wehrheim \cite[Theorem 6.1]{Wehrheim_2004}.)
Let $d\geq 2$, and $G$ be a compact Lie group, and $p \in (1,\infty)$ obeying $d/2 \leq p < d$ and $s_0>1$ be constants. Then there are constants, $\eps=\eps(d,G,p,s_0) \in (0,1]$ and $C=C(d,G,p,s_0) \in [1,\infty)$, with the following significance. For $q \in [p,\infty)$, let $A$ be a $W^{1,q}$ connection on $B\times G$ such that
\begin{equation}
\label{eq:Ldover2_ball_curvature_small}
\|F_A\|_{L^{s_0}(B)} \leq \eps,
\end{equation}
where $B \subset \RR^d$ is the unit ball with center at the origin and $s_0=d/2$ when $d\geq 3$ and $s_0 > 1$ when $d=2$. Then there is a $W^{2,q}$ gauge transformation, $u:B\to G$, such that the following holds. If $A = \Theta + a$, where $\Theta$ is the product connection on $B\times G$, and $u(A) = \Theta + u^{-1}au + u^{-1}du$, then
\begin{align}
\label{eq:Uhlenbeck_local_Coulomb_gauge} 
d^*(u(A) - \Theta) &= 0 \quad \text{a.e. on } B,
 \\
\label{eq:Uhlenbeck_Neumann_boundary_condition}   
(u(A) - \Theta)(\vec n) &= 0 \quad \text{on } \partial B,
\end{align}
where $\vec n$ is the outward-pointing unit normal vector field on $\partial B$, and
\begin{equation}
\label{eq:Uhlenbeck_1-3_W1p_norm_connection_one-form_leq_constant_Lp_norm_curvature}
\|u(A) - \Theta\|_{W^{1,p}(B)} \leq C\|F_A\|_{L^p(B)}.
\end{equation}
\end{thm}

\begin{rmk}[Restriction of $p$ to the range $1<p<\infty$]
\label{rmk:Restriction_p_range}
(See Feehan \cite[Remark 2.6]{Feehan_lojasiewicz_inequality_ground_state_v1}.)  
The restriction $p\in(1,\infty)$ should be included in the statements of \cite[Theorem 1.3 or Theorem 2.1 and Corollary 2.2]{UhlLp} since the bound \eqref{eq:Uhlenbeck_1-3_W1p_norm_connection_one-form_leq_constant_Lp_norm_curvature} ultimately follows from an \apriori $L^p$ estimate for an elliptic system that is apparently only valid when $1<p<\infty$. Wehrheim makes a similar observation in her \cite[Remark 6.2 (d)]{Wehrheim_2004}. This is also the reason that when $d=2$, we require $s_0>1$ in \eqref{eq:Ldover2_ball_curvature_small}.
\end{rmk}

\begin{rmk}[Dependencies of the constants in Theorem \ref{thm:Uhlenbeck_Lp_1-3}]
\label{rmk:Uhlenbeck_Lp_1-3_constant_dependencies}
The statements of \cite[Theorem 1.3 or Theorem 2.1 and Corollary 2.2]{UhlLp} imply that the constants, $\eps$ and $C$, in estimate \eqref{eq:Uhlenbeck_1-3_W1p_norm_connection_one-form_leq_constant_Lp_norm_curvature} only depend the dimension, $d$. However, their proofs suggest that these constants may also depend on $G$ and $p$ through the appeal to an elliptic estimate for $d+d^*$ in the verification of \cite[Lemma 2.4]{UhlLp} and arguments immediately following.
\end{rmk}

\begin{rmk}[Construction of a $W^{k+1,p}$ transformation to Coulomb gauge]
\label{rmk:Uhlenbeck_theorem_1-3_Wkp}
We note that if $A$ is of class $W^{k,p}$, for an integer $k \geq 1$ and $p \geq 2$, then the gauge transformation, $u$, in Theorem \ref{thm:Uhlenbeck_Lp_1-3} is of class $W^{k+1,p}$; see \cite[page 32]{UhlLp}, the proof of \cite[Lemma 2.7]{UhlLp} via the Implicit Function Theorem for smooth functions on Banach spaces, and our proof of \cite[Theorem 1.1]{FeehanSlice} --- a global version of Theorem \ref{thm:Uhlenbeck_Lp_1-3}.
\end{rmk}

Note that if the connection, $A$, in Theorem \ref{thm:Uhlenbeck_Lp_1-3} is flat, then both $F_A \equiv 0$ and $F_{u(A)} \equiv 0$ on $B$, so $u(A) = \Theta$ and thus $A$ is gauge-equivalent to the product connection on $B\times G$. (This conclusion can also be deduced from \cite[Theorem 2.2.1]{DK}.)

\begin{rmk}[Non-flat Riemannian metrics]
\label{rmk:Non-flat_Riemannian_metrics_local_Coulomb_gauge}
(See Feehan \cite[Remark 2.9]{Feehan_lojasiewicz_inequality_ground_state_v1}.)    
Theorem \ref{thm:Uhlenbeck_Lp_1-3} continues to hold for geodesic unit balls in a manifold $X$ endowed with a non-flat Riemannian metric $g$. The only difference in this more general situation is that the constants $C$ and $\eps$ will depend on bounds on the Riemann curvature tensor, $\Riem$. See Wehrheim \cite[Theorem 6.1]{Wehrheim_2004}.
\end{rmk}

We now recall an extension of Theorem \ref{thm:Uhlenbeck_Lp_1-3} to include the range $1 < p < d/2$.

\begin{cor}[Existence of a local Coulomb gauge and \apriori $W^{1,p}$ estimate for a Sobolev connection with $L^{d/2}$-small curvature when $p < d/2$]
\label{cor:Uhlenbeck_theorem_1-3_p_lessthan_dover2}
(See Feehan \cite[Corollary 2.10]{Feehan_lojasiewicz_inequality_ground_state_v1}.)  
Assume the hypotheses of Theorem \ref{thm:Uhlenbeck_Lp_1-3}, but allow any $p \in (1,\infty)$ obeying $p < d/2$ when $d \geq 3$. Then the estimate \eqref{eq:Uhlenbeck_1-3_W1p_norm_connection_one-form_leq_constant_Lp_norm_curvature} holds for $1 < p < d/2$.
\end{cor}

For completeness, we also recall the following extension of Theorem \ref{thm:Uhlenbeck_Lp_1-3} (and slight improvement of our \cite[Corollary 4.4]{Feehan_yangmillsenergygapflat_aim}) to include the range $d \leq p < \infty$.

\begin{cor}[Existence of a local Coulomb gauge and \apriori $W^{1,p}$ estimate for a Sobolev connection one-form with $L^{\bar p}$-small curvature when $p \geq d$]
\label{cor:Uhlenbeck_theorem_1-3_p_geq_d}
(See Feehan \cite[Corollary 2.11]{Feehan_lojasiewicz_inequality_ground_state_v1}.)    
Assume the hypotheses of Theorem \ref{thm:Uhlenbeck_Lp_1-3}, but consider $d \leq p < \infty$ and strengthen \eqref{eq:Ldover2_ball_curvature_small} to\footnote{In \cite[Corollary 4.4]{Feehan_yangmillsenergygapflat_aim}, we assumed the still stronger condition, $\|F_A\|_{L^p(B)} \leq \eps$. }
\begin{equation}
\label{eq:Lbarp_ball_curvature_small}
\|F_A\|_{L^{\bar p}(B)} \leq \eps,
\end{equation}
where $\bar p = dp(d+p)$ when $p>d$ and $\bar p > d/2$ when $p=d$. Then the estimate \eqref{eq:Uhlenbeck_1-3_W1p_norm_connection_one-form_leq_constant_Lp_norm_curvature} holds for $d \leq p < \infty$ and constant $C = C(d,p,\bar p,G) \in [1,\infty)$.
\end{cor}

Taken together, Corollaries \ref{cor:Uhlenbeck_theorem_1-3_p_lessthan_dover2} and \ref{cor:Uhlenbeck_theorem_1-3_p_geq_d} correct and replace our our \cite[Corollary 4.4]{Feehan_yangmillsenergygapflat_aim} (which should have included the restriction $p>1$ when $d=2$). However, neither Theorem \ref{thm:Uhlenbeck_Lp_1-3} nor Corollaries \ref{cor:Uhlenbeck_theorem_1-3_p_lessthan_dover2} and \ref{cor:Uhlenbeck_theorem_1-3_p_geq_d} play a direct role in our corrected proof of Theorem \ref{mainthm:Main} in this corrigendum.

\subsection{A priori estimate for the curvature of a Yang--Mills connection}
\label{subsec:Apriori_estimate_curvature_Yang-Mills_connection}
We next recall the

\begin{thm}[\Apriori interior estimate for the curvature of a Yang--Mills connection]
\label{thm:Uhlenbeck_3-5}
(See \cite[Theorem 3.5]{UhlRem}.)
If $d\geq 3$ is an integer, then there are constants, $K_0=K_0(d) \in [1,\infty)$ and $\eps_0=\eps_0(d) \in (0,1]$, with the following significance. Let $G$ be a compact Lie group, $\rho>0$ be a constant, and $A$ be a Yang--Mills connection with respect to the standard Euclidean metric on $B_{2\rho}(0)\times G$, where $B_r(x_0) \subset \RR^d$ is the open ball with center at $x_0 \in \RR^d$ and radius $r>0$. If
\begin{equation}
\label{thm:Uhlenbeck_3-5_FA_Ld_over2_small_ball}
\|F_A\|_{L^{d/2}(B_{2\rho}(0))} \leq \eps_0,
\end{equation}
then, for all $B_r(x_0) \subset B_\rho(0)$,
\begin{equation}
\label{thm:Uhlenbeck_3-5_Linfty_norm_FA_leq_constant_L2_norm_FA_ball}
\|F_A\|_{L^{\infty}(B_r(x_0))} \leq K_0r^{-d/2}\|F_A\|_{L^2(B_r(x_0))}.
\end{equation}
\end{thm}

As Uhlenbeck notes in \cite[Section 3, first paragraph]{UhlRem}, Theorem \ref{thm:Uhlenbeck_3-5} continues to hold for geodesic balls in a manifold $X$ endowed with a non-flat Riemannian metric, $g$. The only difference in this more general situation is that the constants $K$ and $\eps$ will depend on bounds on the Riemann curvature tensor, $R$, over $B_{2\rho}(x_0)$ and the injectivity radius at $x_0\in X$. Therefore, by employing a finite cover of $X$ by geodesic balls, $B_\rho(x_i)$, of radius $\rho \in (0, \Inj(X,g)/4]$ and applying Theorem \ref{thm:Uhlenbeck_3-5} to each ball $B_{2\rho}(x_i)$, we obtain a global version.

\begin{cor}[\Apriori estimate for the curvature of a Yang--Mills connection over a closed manifold]
\label{cor:Uhlenbeck_3-5_manifold}
Let $X$ be a closed, smooth manifold of dimension $d\geq 3$ and endowed with a Riemannian metric, $g$. Then there are constants, $K=K(g) \in [1,\infty)$ and $\eps=\eps(g) \in (0,1]$, with the following significance. Let $G$ be a compact Lie group and $A$ be a smooth Yang--Mills connection with respect to the metric, $g$, on a smooth principal $G$-bundle $P$ over $X$. If
\begin{equation}
\label{cor:Uhlenbeck_3-5_FA_Ld_over2_small_manifold}
\|F_A\|_{L^{d/2}(X)} \leq \eps,
\end{equation}
then
\begin{equation}
\label{cor:Uhlenbeck_3-5_Linfty_norm_FA_leq_constant_L2_norm_FA_manifold}
\|F_A\|_{L^{\infty}(X)} \leq K\|F_A\|_{L^2(X)}.
\end{equation}
\end{cor}


The restriction $d \geq 3$ in Theorem \ref{thm:Uhlenbeck_3-5} (and hence Corollary \ref{cor:Uhlenbeck_3-5_manifold}) was not explicitly stated by Uhlenbeck in her \cite[Theorem 3.5]{UhlRem} (although it does appear in her \cite[Corollary 2.9]{UhlRem}). However, the condition $d \geq 3$ can be inferred from Uhlenbeck's proof of \cite[Theorem 3.5]{UhlRem}, in particular through her proof of the required \cite[Lemma 3.3]{UhlRem}, where the exponent $\nu=2d/(d-2)$ is undefined when $d=2$. The restriction $d \geq 3$ also appears in Sibner's proof of her \apriori $L^\infty$ estimate for $|F_A|$ in \cite[Proposition 1.1]{Sibner_1984}, where the necessity of the condition appears in her definition \cite[p. 94]{Sibner_1984} of the positive constant $\gamma_1 := (2d-4)/(d^2C_d)$, with $C_d$ denoting a Sobolev embedding constant in dimension $d$. When $d=2$, the proof of \cite[Theorem 4.1]{Smith_1990} due to Smith implies an \apriori $L^p$ estimate for $|F_A|$ (for $1\leq p<\infty$) that is sufficient for the purposes of our proof of Theorem \ref{mainthm:Main} in the case $d=2$; see Lemma \ref{lem:Smith_theorem_4-1}.

\section{Global existence of a flat connection and a Sobolev distance estimate}
\label{sec:Uhlenbeck_approach_existence_flat_connection_and_apriori_estimate}
In Feehan \cite{Feehan_nonlinear_uhlenbeck_estimate}, we proved the following correction to Uhlenbeck's \cite[Corollary 4.3]{UhlChern} (which we had attempted to reprove as \cite[Theorem 5.1]{Feehan_yangmillsenergygapflat_aim}) and a global version of her Theorem \ref{thm:Uhlenbeck_Lp_1-3}:

\begin{thm}[Existence of a nearby $W^{1,p}$ flat connection on a principal bundle supporting a $W^{1,p}$ connection with $L^p$-small curvature]
\label{thm:Uhlenbeck_Chern_corollary_4-3}
(See Feehan \cite[Theorems 1 and 9]{Feehan_nonlinear_uhlenbeck_estimate} for a more general statement.)
Let $X$ be a closed, smooth manifold of dimension $d\geq 2$ and endowed with a Riemannian metric, $g$, and $G$ be a compact Lie group, and $p \in (d/2, \infty)$. Then there are a constant $\eps=\eps(g,G,p) \in (0,1]$ and, for any $r \in (1,p]$, a constant $C=C(g,G,r) \in [1,\infty)$ with the following significance. Let $A$ be a $W^{1,p}$ connection on a principal $G$-bundle $P$ over $X$. If
\begin{equation}
\label{eq:Curvature_Lp_small}
\|F_A\|_{L^p(X)} \leq \eps,
\end{equation}
then there are a $W^{1,p}$ flat connection $\Gamma$ on $P$, a constant $\nu=\nu(g,G,[\Gamma]) \in (0,1]$, and a $W^{2,p}$ gauge transformation $u \in \Aut(P)$ such that
\begin{align}
\label{eq:Uhlenbeck_Chern_corollary_4-3_uA-Gamma_global_Coulomb_gauge}
d_\Gamma^*(u(A) - \Gamma) &= 0 \quad\text{a.e. on } X,
\\
\label{eq:Uhlenbeck_Chern_corollary_4-3_uA-Gamma_W1p_bound}
\|u(A)-\Gamma\|_{W_\Gamma^{1,r}(X)} &\leq C\|F_A\|_{L^r(X)}^\nu.
\end{align}
Moreover, if $d\geq 3$ or $d=2$ and
$p>4/3$, then we may assume that $\Gamma$ is $C^\infty$. 
\end{thm}

The Coulomb gauge condition \eqref{eq:Uhlenbeck_Chern_corollary_4-3_uA-Gamma_global_Coulomb_gauge} is implied by Uhlenbeck's proof of \cite[Corollary 4.3]{UhlChern}, but was not explicitly stated there. The difference between our Theorem \ref{thm:Uhlenbeck_Chern_corollary_4-3} and Uhlenbeck's \cite[Corollary 4.3]{UhlChern} is that we only assert that \eqref{eq:Uhlenbeck_Chern_corollary_4-3_uA-Gamma_W1p_bound} holds for some $\nu(g,G,[\Gamma]) \in (0,1]$ and not necessarily for $\nu=1$. In Appendix \ref{subsec:Estimate_Sobolev_distance_flat_connection_Hodge_Laplacian_vanishing_kernel}, we give a proof that \eqref{eq:Uhlenbeck_Chern_corollary_4-3_uA-Gamma_W1p_bound} holds with $\nu=1$ in the special case where $\Ker\Delta_\Gamma\cap \Omega^1(X;\ad P) = \{0\}$, where we assume that $\Gamma$ is $C^\infty$ for simplicity. More generally (see \cite[Theorem 9]{Feehan_nonlinear_uhlenbeck_estimate}), if the Yang--Mills energy function $\sE$ in \eqref{eq:Yang-Mills_energy_functional} is \emph{Morse--Bott} at the point $[\Gamma]$ in the moduli space of flat connections $M(P)$ in the sense that
\[
  \bU_\Gamma(\delta) := \Gamma+\left\{a\in\Ker d_\Gamma^*\cap \Omega^1(X;\ad P):\|a\|_{W_\Gamma^{1,p}(X)} < \delta \text{ and } F_{\Gamma+a}=0\right\}
\]
is a smooth manifold for small enough $\delta=(g,G,p,\Gamma)\in(0,1]$ and
\[
  T_\Gamma \bU_\Gamma(\delta) = \Ker\sE''(\Gamma),
\]
where $\sE''(\Gamma) = \Ker\Delta_\Gamma$ is the Hessian operator on $\Omega^1(X;\ad P)$, then \eqref{eq:Uhlenbeck_Chern_corollary_4-3_uA-Gamma_W1p_bound} also holds with $\nu=1$.

Donaldson and Kronheimer \cite[Proposition 4.4.11]{DK} employ the local Coulomb-gauge \eqref{eq:Uhlenbeck_1-3_W1p_norm_connection_one-form_leq_constant_Lp_norm_curvature} and a patching argument to prove that \eqref{eq:Uhlenbeck_Chern_corollary_4-3_uA-Gamma_W1p_bound} holds with $\nu=1$ when $X$ is \emph{strongly simply connected} and $p=2$ and $d=2,3$ and $\Gamma=\Theta$ but remark \cite[p. 163]{DK} that their result extends to $d = 4$ and $p>2$. In \cite[Proposition 4.4.11]{DK}, it is not claimed that $d_\Theta^*(u(A)-\Theta)=0$. Recall that $X$ is strongly simply connected \cite[p. 161]{DK} if it can be covered by smoothly embedded balls $B_1,\ldots,B_m$ such that for any $2 \leq r \leq m$, the intersection $B_r \cap (B_1\cup\cdots\cup B_{r-1})$ is connected; the condition implies that $X$ is simply connected.

Fukaya \cite[Proposition 3.1]{Fukaya_1998} proved that a version\footnote{Fukaya uses a different system of norms.} of \eqref{eq:Uhlenbeck_Chern_corollary_4-3_uA-Gamma_W1p_bound} holds when $d=4$, and $X$ is a compact manifold with boundary, and $A$ is anti-self-dual. Fukaya's proof of \cite[Proposition 3.1]{Fukaya_1998} uses the local Coulomb-gauge \eqref{eq:Uhlenbeck_1-3_W1p_norm_connection_one-form_leq_constant_Lp_norm_curvature} and difficult patching argument. In \cite[Proposition 3.1]{Fukaya_1998}, it is not claimed that $d_\Gamma^*(u(A)-\Gamma)=0$. He explained to us \cite{Fukaya_9-25-2018} that his proof should extend to allow arbitrary dimensions $d\geq 2$, connections $A$ on $P$ of Sobolev class $W^{1,p}$, and the system of Sobolev norms in \eqref{eq:Uhlenbeck_Chern_corollary_4-3_uA-Gamma_W1p_bound}.  If $X$ is a compact manifold without boundary and $A$ is anti-self-dual and $\|F_A\|_{L^2(X)}$ is smaller than a constant that depends at most on $G$, then the Chern--Weil Theorem (see Milnor and Stasheff \cite[Appendix C]{MilnorStasheff}) would imply that $A$ is necessarily flat.

Nishinou \cite{Nishinou_2007} proved that a version\footnote{Nishinou uses a different system of norms.} of \eqref{eq:Uhlenbeck_Chern_corollary_4-3_uA-Gamma_W1p_bound} holds when $X=\TT^2$ (the real two-dimensional torus) and $P = \TT^2\times\SU(2)$ and and $\Gamma$ is the product connection and $\nu = 1/2$.

In \cite[Appendix A]{Feehan_nonlinear_uhlenbeck_estimate}, we develop an example due to Mrowka \cite{Mrowka_7-30-2018} which shows that \eqref{eq:Uhlenbeck_Chern_corollary_4-3_uA-Gamma_W1p_bound} cannot hold for $\nu>1/2$ when $A_t$, for $t\in (-\delta,\delta)$, is a certain family of smooth connections on $\TT^2\times\SU(2)$ in Coulomb gauge with respect to the product connection $\Theta$. There is no other flat connection $\Gamma$ on $\TT^2\times\SU(2)$ that is in Coulomb gauge with respect to $\Theta$ and closer in the $W^{1,p}$ norm to $A_t$, for $t\in (-\delta,\delta)$, and that is no gauge transformation $u\in\Aut(P)$ that can be used to improve the estimate \eqref{eq:Uhlenbeck_Chern_corollary_4-3_uA-Gamma_W1p_bound} by replacing $A_t$ by $u(A_t)$.

The argument provided by Uhlenbeck in her proof of \cite[Corollary 4.3]{UhlChern} was very brief and that prompted us to attempt a more detailed justification in \cite{Feehan_yangmillsenergygapflat_aim, Feehan_yangmillsenergygapflat_arxiv_v7} using using the local Coulomb-gauge \eqref{eq:Uhlenbeck_1-3_W1p_norm_connection_one-form_leq_constant_Lp_norm_curvature} in Theorem \ref{thm:Uhlenbeck_Lp_1-3} and a patching argument, which is incorrect, as we explain in \cite{Feehan_yangmillsenergygapflat_corrigendum}. 

The estimate \eqref{eq:Uhlenbeck_Chern_corollary_4-3_uA-Gamma_W1p_bound} may be expressed in a more invariant way that is also more suggestive of the relevance of versions of the {\L}ojasiewicz--Simon gradient inequality (compare S.-Z. Huang \cite[Theorem 2.3.1 (i)]{Huang_2006}, {\L}ojasiewicz \cite{Lojasiewicz_1963, Lojasiewicz_1965, Lojasiewicz_1993}, and Simon \cite[Equations (2.1) and (2.3)]{Simon_1983}),
\[
\dist_{W^{1,p}(X)}\left([A], M(P)\right) \leq C\|F_A\|_{L^p(X)}^\nu,
\]
where
\[
\dist_{W^{1,p}(X)} \left([A], M(P)\right)
:=
\inf_{\begin{subarray}{c} u \in \Aut(P), \\ [\Gamma] \in M(P) \end{subarray}}
\|u(A) - \Gamma\|_{W^{1,p}(X)}.
\]
The proof of the existence of the flat connection $\Gamma$ in Theorem \ref{thm:Uhlenbeck_Chern_corollary_4-3} follows by standard arguments (see \cite{DK, FU}), though we included the details for completeness in the proof of our \cite[Theorem 1]{Feehan_nonlinear_uhlenbeck_estimate}.

\begin{rmk}[On Theorem \ref{thm:Uhlenbeck_Chern_corollary_4-3} for vector bundles]
\label{rmk:Uhlenbeck_Chern_corollary_4-3_for_vector_bundles}
For the sake of consistency with the remainder of our article, we have converted
\cite[Corollary 4.3]{UhlChern} to the equivalent setting of a principal $G$-bundle rather than its original setting of a vector bundle $E$ with compact Lie structure group, $G$, and an orthogonal representation, $G\hookrightarrow \SO(l)$, for some integer $l\geq 2$ as in Uhlenbeck \cite[Section 1]{UhlRem}.
\end{rmk}

\section{Completion of the proof of Theorem \ref{mainthm:Main}}
\label{sec:Completion_proof_main_theorem}
We first note the following consequence of Corollary \ref{cor:Uhlenbeck_3-5_manifold} and Theorem \ref{thm:Uhlenbeck_Chern_corollary_4-3}.

\begin{cor}[Existence of a nearby flat connection on a principal bundle supporting a $C^\infty$ Yang--Mills connection with $L^{d/2}$-small curvature]
\label{cor:Uhlenbeck_Chern_A_Yang-Mills_FA_Ldover2_small}
Let $X$ be a closed, smooth manifold of dimension $d\geq 2$ and endowed with a Riemannian metric, $g$, and $G$ be a compact Lie group, and $p \in (d/2, \infty)$. Then there are a constant $\eps=\eps(g,G,p) \in (0,1]$ and, for any $r \in (1,p]$, a constant $C=C(g,G,r) \in [1,\infty)$ with the following significance. Let $A$ be a $C^\infty$ Yang--Mills connection on a $C^\infty$ principal $G$-bundle $P$ over $X$. If the curvature $F_A$ obeys \eqref{eq:Curvature_Ldover2_small}, that is,
\[
\|F_A\|_{L^{d/2}(X)} \leq \eps,
\]
then there are a $C^\infty$ flat connection $\Gamma$ on $P$, a constant $\nu=\nu(g,G,[\Gamma]) \in (0,1]$, and a $C^\infty$ gauge transformation $u \in \Aut(P)$ such that
\begin{align}
\label{eq:Uhlenbeck_Chern_A_Yang-Mills_uA-Gamma_global_Coulomb_gauge}
d_\Gamma^*(u(A) - \Gamma) &= 0 \quad\text{a.e. on } X,
\\
\label{eq:Uhlenbeck_Chern_A_Yang-Mills_uA-Gamma_W1p_bound}
\|u(A)-\Gamma\|_{W_\Gamma^{1,r}(X)} &\leq C\|F_A\|_{L^r(X)}^\nu.
\end{align}
\end{cor}

\begin{proof}
For any $d\geq 3$ or $d=2$ and $p \geq 1$, the estimates \eqref{cor:Uhlenbeck_3-5_Linfty_norm_FA_leq_constant_L2_norm_FA_manifold} in Corollary \ref{cor:Uhlenbeck_3-5_manifold} and \eqref{eq:Uhlenbeck_3-5_Lp_norm_FA_leq_constant_L1_norm_FA_manifold} in Corollary \ref{cor:Smith_theorem_4-1_manifold}, respectively, yield
\begin{subequations}
\label{eq:LpnormFA_leq_constant_L2normFA}
\begin{align}
\label{eq:LpnormFA_leq_constant_L2normFA_d_geq_3}
\|F_A\|_{L^p(X)} &\leq \left(\Vol_g(X)\right)^{1/p}\|F_A\|_{L^\infty(X)}
\leq K\left(\Vol_g(X)\right)^{1/p}\|F_A\|_{L^2(X)} \quad (d \geq 3),
\\
\label{eq:LpnormFA_leq_constant_L2normFA_d_is_2}
\|F_A\|_{L^p(X)} &\leq K_p\|F_A\|_{L^1(X)} = K_p\|F_A\|_{L^{d/2}(X)} \quad (d = 2),
\end{align}
\end{subequations}
for $K = K(g) \in [1,\infty)$ and $K_p = K_p(g,p) \in [1,\infty)$. If $d > 4$, then (writing $1/2 = (d-4)/(2d) + 2/d$)
\begin{equation}
\label{eq:L2normFA_leq_constant_Ldover2normFA_d_bigger_4}
\|F_A\|_{L^2(X)} \leq \left(\Vol_g(X)\right)^{2d/(d-4)}\|F_A\|_{L^{d/2}(X)}, \quad\forall\, d \geq 5.
\end{equation}
If $d=3$, then $L^p$ interpolation \cite[Equation (7.9)]{GilbargTrudinger} implies that
\[
\|F_A\|_{L^2(X)} \leq \|F_A\|_{L^{3/2}(X)}^\lambda \|F_A\|_{L^r(X)}^{1-\lambda},
\]
where the exponent $r$ obeys $2 < r \leq \infty$ and the constant $\lambda \in (0,1)$ is defined by $1/2 = \lambda/(3/2) + (1-\lambda)/r$. We may choose $r = \infty$ and thus $\lambda = 3/4$ to give
\begin{align*}
\|F_A\|_{L^2(X)} &\leq \|F_A\|_{L^{3/2}(X)}^{3/4} \|F_A\|_{L^\infty(X)}^{1/4}
\\
&\leq \|F_A\|_{L^{3/2}(X)}^{3/4} \left(K\|F_A\|_{L^2(X)}\right)^{1/4}
\quad\text{(by Corollary \ref{cor:Uhlenbeck_3-5_manifold})},
\end{align*}
and thus
\begin{equation}
\label{eq:L2normFA_leq_constant_Ldover2normFA_d_leq_4}
\|F_A\|_{L^2(X)} \leq K^{(4-d)/d}\|F_A\|_{L^{d/2}(X)}, \quad d=3,4.
\end{equation}
Therefore, by combining \eqref{eq:LpnormFA_leq_constant_L2normFA} (for $d\geq 2$), \eqref{eq:L2normFA_leq_constant_Ldover2normFA_d_bigger_4} (for $d\geq 5$), and \eqref{eq:L2normFA_leq_constant_Ldover2normFA_d_leq_4} (for $d=3,4$), we obtain
\begin{equation}
\label{eq:LpnormFA_leq_constant_Ldover2normFA_d_geq_2_and_p_geq_1}
\|F_A\|_{L^p(X)} \leq C_1\|F_A\|_{L^{d/2}(X)}, \quad\forall\, d \geq 2 \text{ and } p \geq 1,
\end{equation}
for $C_1 = C_1(g,p) \in [1,\infty)$. Hence, the preceding inequality and the hypothesis \eqref{eq:Curvature_Ldover2_small}, namely $\|F_A\|_{L^{d/2}(X)} \leq \eps$, of Corollary \eqref{cor:Uhlenbeck_Chern_A_Yang-Mills_FA_Ldover2_small} ensure that the hypothesis \eqref{eq:Curvature_Lp_small} of Theorem \ref{thm:Uhlenbeck_Chern_corollary_4-3} applies for small enough $\eps=\eps(g,G) \in (0,1]$ by taking $p = (d+1)/2$ in  \eqref{eq:Curvature_Lp_small}. The conclusions now follow from Theorem \ref{thm:Uhlenbeck_Chern_corollary_4-3} and Remark \ref{rmk:Uhlenbeck_theorem_1-3_Wkp} for smoothness of $u$. 
\end{proof}

\begin{rmk}[Stronger form of the estimate \eqref{eq:Uhlenbeck_Chern_A_Yang-Mills_uA-Gamma_W1p_bound}]
\label{rmk:Uhlenbeck_Chern_A_Yang-Mills_FA_Ldover2_small_estimate_scale_invariance}
The estimate \eqref{eq:Uhlenbeck_Chern_A_Yang-Mills_uA-Gamma_W1p_bound} has been left in the invariant form provided by Theorem \ref{thm:Uhlenbeck_Chern_corollary_4-3}, but it could easily be improved (by replacing $\|F_A\|_{L^p(X)}$ on the right-hand side with $\|F_A\|_{L^2(X)}$ when $p > 2$) with the aid of Corollary \ref{cor:Uhlenbeck_3-5_manifold} since $A$ is a Yang--Mills connection.
\end{rmk}

We can now finally complete the

\begin{proof}[Proof of Theorem \ref{mainthm:Main}]
For small enough $\eps = \eps(g,G) \in (0,1]$, Corollary \ref{cor:Uhlenbeck_Chern_A_Yang-Mills_FA_Ldover2_small} provides a smooth flat connection $\Gamma$ on $P$, a constant $\nu = \nu(g,G,[\Gamma]) \in (0,1]$, and a smooth gauge transformation $u\in\Aut(P)$, and the estimate
\[
\|u(A) - \Gamma\|_{W_\Gamma^{1,p}(X)} \leq C_0\|F_A\|_{L^p(X)}^\nu,
\]
for $p = \min\{2, d/2\}$ and $C_0 = C_0(g,G) \in [1,\infty)$. The preceding inequality ensures that the hypothesis \eqref{eq:Rade_7-1_neighborhood_flat} of Corollary \ref{cor:Rade_proposition_7-2_flat} holds,
\[
\|u(A) - \Gamma\|_{W^{1,p}_\Gamma(X)} < \sigma,
\]
provided, for example, $\|F_A\|_{L^p(X)} \leq (\sigma/(2C_0))^{1/\nu}$. The latter condition is ensured in turn by the hypothesis \eqref{eq:Curvature_Ldover2_small}, namely $\|F_A\|_{L^{d/2}(X)} \leq \eps$, of Theorem \ref{mainthm:Main} for small enough $\eps = \eps(g,G) \in (0,1]$, since \eqref{eq:LpnormFA_leq_constant_L2normFA_d_is_2} and \eqref{eq:L2normFA_leq_constant_Ldover2normFA_d_leq_4} give
\[
\|F_A\|_{L^2(X)} \leq C_1\|F_A\|_{L^{d/2}(X)}, \quad\text{for } d=2,3,
\]
for $C_1 = C_1(g) \in [1, \infty)$. Indeed, the constant
\[
\eps := \begin{cases} \sigma/(2C_0) &\text{for } d \geq 4, \\ \sigma/(2C_0C_1) &\text{for } d=2,3, \end{cases}
\]
will suffice. If $p' = p/(p-1) \in (1,2]$ is the H\"older exponent dual to $p \in [2,\infty)$, then the Sobolev Embedding \cite[Theorem 4.12]{AdamsFournier} (for $d \geq 2$) implies that $W^{1,p'}(X) \subset L^r(X)$ is a continuous embedding if
\begin{inparaenum}[(\itshape i\upshape)]
\item $1 < p' < d$ and $1 < r = (p')^* := dp'/(d-p') \in (1,\infty)$, or
\item $p' = d$ and $1 < r < \infty$, or
\item $d < p' < \infty$ and $r = \infty$.
\end{inparaenum}
Since $d\geq 2$ by hypothesis, only the first two cases can occur and by duality and density, we obtain a continuous Sobolev embedding, $L^{r'}(X) \subset W^{-1,p}(X)$, where $r' = r/(r-1) \in (1,\infty)$ is the H\"older exponent dual to $r \in (1,\infty)$. The Kato Inequality \cite[Equation (6.20)]{FU} implies that the norm of the induced Sobolev embedding, $W_\Gamma^{1,p'}(X;\Lambda^1\otimes\ad P) \subset L^r(X;\Lambda^1\otimes\ad P)$, is independent of $\Gamma$, and hence the norm, $\kappa = \kappa(g,p) \in [1,\infty)$ of the dual Sobolev embedding, $L^{r'}(X;\Lambda^1\otimes\ad P) \subset W_\Gamma^{-1,p}(X;\Lambda^1\otimes\ad P)$, is also independent of $\Gamma$. The preceding embedding and the {\L}ojasiewicz--Simon gradient inequality, Corollary \ref{cor:Rade_proposition_7-2_flat} applied to $u(A)$, now yield
\[
\|d_A^*F_A\|_{L^{r'}(X)} \geq \kappa^{-1}c|\sE(A)|^\theta,
\]
noting that each side of the inequality \eqref{eq:Rade_7-1_flat} remains unchanged when $u(A)$ is interchanged with $A$. But $A$ is a Yang--Mills connection, so $d_A^*F_A = 0$ on $X$ and $\sE(A) = \frac{1}{2}\|F_A\|_{L^2(X)} = 0$ by \eqref{eq:Yang-Mills_energy_functional} and thus $A$ must be a flat connection.
\end{proof}

\appendix

\section{Alternative proofs under simplifying hypotheses}
\label{app:Alternative_proofs_simplifying_hypotheses}
The proofs of several results described in this article simplify considerably under the assumption of additional hypotheses. We discuss these simpler proofs in this Appendix.

\subsection{Bochner--Weitzenb\"ock formula and existence of an $L^{d/2}$-energy gap}
\label{subsec:BW_formula_existence_Ldover2_energy_gap}
When the articles by Bourguignon, Lawson, and Simons
\cite{Bourguignon_Lawson_1981, Bourguignon_Lawson_Simons_1979} were developed, the \apriori estimate (Theorem \ref{thm:Uhlenbeck_3-5}) due to Uhlenbeck for the curvature, $F_A$, of Yang--Mills connection, $A$, had not yet been published. In particular, their energy gap results are phrased in terms of $L^\infty$ rather than $L^{d/2}$-small enough curvature, $F_A$. However, \apriori estimates for Yang--Mills connections were incorporated by Donaldson and Kronheimer in the proof of their $L^2$-energy gap result \cite[Lemma 2.3.24]{DK} for a Yang--Mills connection over the four-dimensional sphere, $S^4$, with its standard round metric of radius one. In this subsection, we describe the minor modifications required to extend their result to the case of a closed Riemannian manifold, $X$, of dimension $d\geq 2$ and whose curvature obeys the positivity condition \eqref{eq:Bourguignon_Lawson_1981_equation_5-1_BW_Riemann_curvature_positivity}. This result, with a proof that is somewhat different to that of \cite[Lemma 2.3.24]{DK} and the argument described here, was provided by Gerhardt in \cite[Theorem 1.2]{Gerhardt_2010}.

\begin{thm}[$L^{d/2}$-energy gap for Yang--Mills connections over Riemannian manifolds with positive curvature]
\label{thm:Donaldson_Kronheimer_Lemma_2-3-24_generalization}
Let $X$ be a closed, smooth manifold of dimension $d \geq 2$ and endowed with a smooth Riemannian metric, $g$, whose curvature obeys \eqref{eq:Bourguignon_Lawson_1981_equation_5-1_BW_Riemann_curvature_positivity}. Then there is a positive constant, $\eps = \eps(d, g) \in (0, 1]$, with the following significance. Let $G$ be a compact Lie group. If $A$ is a smooth connection, on a principal $G$-bundle $P$, that is Yang--Mills with respect to the metric, $g$, and whose curvature, $F_A$, obeys \eqref{eq:Curvature_Ldover2_small}, then $A$ is a flat connection.
\end{thm}

\begin{proof}
We adapt the proof of \cite[Lemma 2.3.24]{DK}, where it is assumed that $X=S^d$ with its standard round metric of radius one and $d=4$; according to \cite[Corollary 3.14]{Bourguignon_Lawson_1981}, one has equality in \eqref{eq:Bourguignon_Lawson_1981_equation_5-1_BW_Riemann_curvature_positivity} with $\lambda_g = 2(d-2)$ when $X=S^d$. This property is noted in \cite[paragraph following Equation (2.3.18)]{DK} for the case $d=4$ and exploited in their proof of \cite[Lemma 2.3.24]{DK}. A closely related positivity result is noted by Gerhardt in \cite[Remark 1.1 (i)]{Gerhardt_2010}.

When $A$ is a Yang--Mills connection, that fact and the Bianchi identity \cite[Equation (2.1.21)]{DK} imply that $\Delta_A F_A = 0$ and so the Bochner--Weitzenb\"ock formula \eqref{eq:Bourguignon_Lawson_1981_Theorem 3-10_and_Equation_5-1_Hodge_Laplacian} yields
\[
(\nabla_A^*\nabla_A F_A, F_A)_{L^2(X)} + (F_A \circ (\Ric_g \wedge I + 2\Riem_g), F_A)_{L^2(X)}
+ (\{F_A, F_A\}, F_A)_{L^2(X)}  = 0.
\]
Therefore,
\[
\|\nabla_A F_A\|_{L^2(X)}^2 + (F_A \circ (\Ric_g \wedge I + 2\Riem_g), F_A)_{L^2(X)}
\leq c\|F_A\|_{L^2(X)}^2 \|F_A\|_{L^\infty(X)}.
\]
If \eqref{eq:Bourguignon_Lawson_1981_equation_5-1_BW_Riemann_curvature_positivity} holds, then the preceding inequality simplifies to give
\[
\|\nabla_A F_A\|_{L^2(X)}^2 + \lambda_g\|F_A\|_{L^2(X)}^2 \leq c\|F_A\|_{L^2(X)}^2 \|F_A\|_{L^\infty(X)},
\]
where $c$ is a positive constant depending at most on the Riemannian metric, $g$. But $\|F_A\|_{L^{d/2}(X)} \leq \eps$ by hypothesis \eqref{eq:Curvature_Ldover2_small} and if $\eps  = \eps(g) \in (0,1]$ is sufficiently small, then Corollary \ref{cor:Uhlenbeck_3-5_manifold} and the preceding inequality yield
\[
\lambda_g\|F_A\|_{L^2(X)}^2 \leq cK\|F_A\|_{L^2(X)}^3,
\]
for $K = K(g) \in [1,\infty)$. If $A$ is not flat, then $F_A$ must obey
\begin{equation}
\label{eq:Gerhardt_1-2_L2}
\|F_A\|_{L^2(X)} \geq \frac{\lambda_g}{cK} > 0.
\end{equation}
But from \eqref{eq:LpnormFA_leq_constant_Ldover2normFA_d_geq_2_and_p_geq_1} we have
\[
\|F_A\|_{L^2(X)} \leq C_1\|F_A\|_{L^{d/2}(X)}, \quad\forall\, d \geq 2,
\]
and $C_1 = C_1(g) \in [1,\infty)$. Combining the preceding inequality with \eqref{eq:Gerhardt_1-2_L2} gives
\[
\|F_A\|_{L^{d/2}(X)} \geq \frac{\lambda_g}{cC_1K} > 0,
\]
and hence a contradiction to \eqref{eq:Curvature_Ldover2_small} for small enough $\eps = \eps(g) \in (0,1]$ in \eqref{eq:Curvature_Ldover2_small}. This completes the proof of Theorem \ref{thm:Donaldson_Kronheimer_Lemma_2-3-24_generalization}.
\end{proof}

\subsection{Estimate of Sobolev distance to the flat connection when $\Ker\Delta_\Gamma = 0$}
\label{subsec:Estimate_Sobolev_distance_flat_connection_Hodge_Laplacian_vanishing_kernel}
It is illuminating to prove the estimate \eqref{eq:Uhlenbeck_Chern_corollary_4-3_uA-Gamma_W1p_bound} in Theorem \ref{thm:Uhlenbeck_Chern_corollary_4-3} in the simplest case, when
\begin{equation}
\label{eq:Hodge_Laplacian_flat_connection_vanishing_kernel}
\Ker\Delta_\Gamma \cap \Omega^1(X;\ad P) = 0,
\end{equation}
although not required in our article. The general case, when $\Ker\Delta_\Gamma \neq 0$, is more difficult and is proved independently by a quite different method in our article \cite{Feehan_nonlinear_uhlenbeck_estimate}. We first recall a simpler version of our correction to \cite[Corollary 4.3]{UhlChern} that is much easier to prove than our Theorem \ref{thm:Uhlenbeck_Chern_corollary_4-3} but yields a less powerful conclusion.

\begin{thm}[Existence of a flat connection on a principal bundle supporting a connection with $L^p$-small curvature]
\label{thm:Uhlenbeck_Chern_corollary_4-3_prelim}
(See Feehan \cite[Theorem 1]{Feehan_nonlinear_uhlenbeck_estimate}.)
Let $G$ be a compact Lie group, $P$ be a smooth principal $G$-bundle over a smooth Riemannian manifold $(X,g)$ of dimension $d \geq 2$, and $p \in (d/2,\infty)$, and $A_1$ be a $C^\infty$ reference connection on $P$. Then there is a constant $\eps=\eps(A_1,g,G,p,[P]) \in (0,1]$ with the following significance. If $A$ is a $W^{1,p}$ connection on $P$ with
\begin{equation}
\label{eq:Lp_norm_FA_lessthan_epsilon}
\|F_A\|_{L^p(X)} \leq \eps,
\end{equation}
then there are a $W^{1,p}$ flat connection $\Gamma$ on $P$ and a $W^{2,p}$ gauge transformation $v \in \Aut^{2,p}(P)$ such that
  \begin{align}
    \label{eq:Uhlenbeck_Chern_corollary_4-3_uA-Gamma_global_Coulomb_gauge_prelim}
    d_\Gamma^*(v(A) - \Gamma) &= 0 \quad\text{a.e. on } X,
    \\
    \label{eq:Uhlenbeck_Chern_corollary_4-3_uA-Gamma_W1p_bound_prelim}
    \|v(A)-\Gamma\|_{W_{A_1}^{1,r}(X)} &\leq C\|F_A\|_{L^r(X)} + C\|v(A)-\Gamma\|_{L^r(X)},
  \end{align}
for any $r \in (1,p]$ and corresponding constant $C=C(A_1,g,G,r) \in [1,\infty)$. If $d\geq 3$ or $d=2$ and
$p>4/3$, then we may assume that $\Gamma$ is $C^\infty$. If $\zeta \in (0,1]$ is a constant then, for a possibly smaller $\eps=\eps(A_1,g,G,p,[P],\sigma) \in (0,1]$, one has
    \begin{equation}
    \label{eq:Uhlenbeck_compactness_bound}
    \|v(A)-\Gamma\|_{W_{A_1}^{1,p}(X)} < \sigma.
  \end{equation}
\end{thm}

We can now proceed to the

\begin{proof}[Proof of the estimate \eqref{eq:Uhlenbeck_Chern_corollary_4-3_uA-Gamma_W1p_bound} in Theorem \ref{thm:Uhlenbeck_Chern_corollary_4-3} when $\Ker\Delta_\Gamma = 0$]
The first part of the proof of Theorem \ref{thm:Uhlenbeck_Chern_corollary_4-3_prelim} (which just relies on \cite[Theorem 1.5 or 3.6]{UhlLp} due to Uhlenbeck) yields a $W^{2,p}$ gauge transformation, $v \in \Aut(P)$, and a flat connection, $\Gamma$, on $P$ such that
\[
\|v(A) - \Gamma\|_{L^q(X)} \leq \zeta \quad\text{and}\quad \|v(A) - \Gamma\|_{W_\Gamma^{1,p}(X)} \leq C,
\]
where $d/2 < p < \infty$ and $d < q < 2p$ and $\zeta = \delta(g,G,q, \eps) \in (0,1]$ is small and $C = C(g,G,p) \in [1,\infty)$ is finite. For notational convenience, we may therefore assume
\[
\|A - \Gamma\|_{L^q(X)} \leq \zeta \quad\text{and}\quad \|A - \Gamma\|_{W_\Gamma^{1,p}(X)} \leq C.
\]
By \eqref{eq:Uhlenbeck_Chern_corollary_4-3_uA-Gamma_global_Coulomb_gauge} in Theorem \ref{thm:Uhlenbeck_Chern_corollary_4-3}, there exists a $W^{2,p}$ gauge transformation, $u\in \Aut(P)$, such that
\[
d_\Gamma^*(u(A)-\Gamma) = 0 \quad\text{a.e. on } X.
\]
As usual, while the hypothesis $p>d/2$ is required for the existence of a global gauge transformation, $u \in \Aut(P)$, since $W^{2,p}(X) \subset C(X)$ is a continuous embedding for $p > d/2$ but not for $p=d/2$, we shall see that --- given the existence of $u \in \Aut(P)$ already --- the estimate \eqref{eq:Uhlenbeck_Chern_corollary_4-3_uA-Gamma_W1p_bound} is valid for any $p$ in the range $d/2 \leq p < \infty$. Hence, for the remainder of the proof we allow $d/2 \leq p < \infty$ and separately consider the cases
\begin{inparaenum}[(\itshape i\upshape)]
\item $d/2 \leq p < d$,
\item $p=d$, and
\item $p > d$,
\end{inparaenum}
with one \emph{caveat} when $p=d/2$: this case still requires a hypothesis $\|F_A\|_{L^{p_0}(X)} \leq \eps$ for some $p_0 > d/2$ and thus $\|A - \Gamma\|_{L^q(X)} \leq \zeta$ for some $q$ in the range $d < q < 2p_0$.

We first consider the case $d/2 \leq p < d$ and write $u(A) = \Gamma + a$ and $F_{u(A)} = F_\Gamma + d_\Gamma a + a\wedge a$, that is, $F_{u(A)} = d_\Gamma a + a\wedge a$ and so the Coulomb gauge condition, $d_\Gamma^* a=0$ a.e. on $X$, yields the first-order, semi-linear elliptic equation,
\[
(d_\Gamma + d_\Gamma^*)a + a\wedge a = F_{u(A)} \quad\text{a.e. on } X.
\]
By replacing the appeal to the $C^1$ Inverse Function Theorem \cite[Theorem 2.5.2]{AMR} by the $C^2$ Quantitative Inverse Function Theorem \cite[Proposition 2.5.6]{AMR} in the proof of \eqref{eq:Uhlenbeck_Chern_corollary_4-3_uA-Gamma_global_Coulomb_gauge} in Theorem \ref{thm:Uhlenbeck_Chern_corollary_4-3}, we can assume that $a = u(A)-\Gamma$ remains $L^q$-small, say
\[
\|a\|_{L^q(X)} \leq cK_0\zeta,
\]
where\footnote{From Section \ref{subsec:Uhlenbeck_compactness_moduli_space_flat_connections}, one has Uhlenbeck compactness of the moduli space of flat connections on $P$ and so $\mu[\Gamma]$ has a positive lower bound, $\mu_0$, independent of $[\Gamma] \in M(P)$ by continuity of $\mu[\Gamma]$ with respect to $\Gamma$ if \eqref{eq:Hodge_Laplacian_flat_connection_vanishing_kernel} holds for every flat connection, $\Gamma$, on $P$.}
$K_0 = 1 + \mu[\Gamma]^{-1}$ and $\mu[\Gamma] > 0$ is the least eigenvalue of $\Delta_\Gamma$, which is positive by \eqref{eq:Hodge_Laplacian_flat_connection_vanishing_kernel}, and the positive constant, $c$, depends at most on the Riemannian metric, $g$, and Lie group, $G$. (See the proof of \cite[Theorem 8.2]{FeehanSlice}.) The operator,
\[
d_\Gamma + d_\Gamma^*: \Omega^1(X;\ad P) \to \Omega^2(X;\ad P)\oplus \Omega^0(X;\ad P),
\]
is first-order elliptic (since $(d_\Gamma^* + d_\Gamma)(d_\Gamma + d_\Gamma^*) = \Delta_\Gamma$, as $F_\Gamma \equiv 0$) and so one has an \apriori estimate (this is well-known but see \cite[Theorem 14.60]{Feehan_yang_mills_gradient_flow_v4} and references therein),
\[
\|a\|_{W_\Gamma^{1,p}(X)}
\leq
C\left(\|(d_\Gamma + d_\Gamma^*)a\|_{L^p(X)} + \|a\|_{L^p(X)}\right),
\]
for a positive constant, $C = C(g,G,p) \in [1, \infty)$. Therefore,
\[
\|a\|_{W_\Gamma^{1,p}(X)}
\leq
C\left(\|a\wedge a\|_{L^p(X)} + \|F_A\|_{L^p(X)} + \|a\|_{L^p(X)}\right).
\]
Define a Sobolev exponent, $r$, by $1/p = 1/q + 1/r$, where $d \leq q \leq 2p$ and $2p \leq r \leq dp/(d-p)$. The H\"older inequality then yields
\[
\|a\wedge a\|_{L^p(X)} \leq 2\|a\|_{L^q(X)} \|a\|_{L^r(X)}.
\]
Also, there is a continuous Sobolev embedding, $W^{1,p}(X) \subset L^s(X)$, for $1\leq p<d$ and $1 \leq s \leq p_* = dp/(d-p)$ by \cite[Theorem 4.12]{AdamsFournier}, and thus
\[
\|a\|_{L^{p_*}(X)} \leq C_0\|a\|_{W_\Gamma^{1,p}(X)},
\]
for $C_0  = C_0(g,p) \in [1,\infty)$. Since $r \leq p_*$, we have
\[
\|a\|_{L^r(X)} \leq C\|a\|_{W_\Gamma^{1,p}(X)},
\]
for $C = C(g,p,r) \in [1,\infty)$, and therefore
\[
\|a\wedge a\|_{L^p(X)} \leq C\|a\|_{L^q(X)} \|a\|_{W_\Gamma^{1,p}(X)},
\]
for $C = C(g,p,q) \in [1,\infty)$. Hence, we obtain
\[
\|a\|_{W_\Gamma^{1,p}(X)}
\leq
C\left(\|a\|_{L^q(X)} \|a\|_{W_\Gamma^{1,p}(X)} + \|F_A\|_{L^p(X)} + \|a\|_{L^p(X)}\right),
\]
for $C = C(g,G,p,q) \in [1,\infty)$. The assumption \eqref{eq:Hodge_Laplacian_flat_connection_vanishing_kernel} is equivalent to
\[
\Ker \left(d_\Gamma + d_\Gamma^*: \Omega^1(X;\ad P) \to \Omega^2(X;\ad P)\oplus \Omega^0(X;\ad P) \right) = 0,
\]
and so, in this case, the \apriori elliptic estimate for $d_\Gamma + d_\Gamma^*$ simplifies to
\[
\|a\|_{W_\Gamma^{1,p}(X)} \leq C\|(d_\Gamma + d_\Gamma^*)a\|_{L^p(X)},
\]
for $C = C(g,G,p) \in [1,\infty)$. Consequently,
\[
\|a\|_{W_\Gamma^{1,p}(X)}
\leq
C\left(\|a\|_{L^q(X)} \|a\|_{W_\Gamma^{1,p}(X)} + \|F_A\|_{L^p(X)}\right),
\]
for $C = C(g,G,p,q) \in [1,\infty)$. For $0 < \zeta < 1/(2C)$, where $C$ is as in the preceding estimate, we can use the bound $\|a\|_{L^q(X)} \leq \zeta$ and rearrangement to give
\[
\|a\|_{W_\Gamma^{1,p}(X)}
\leq
C\|F_A\|_{L^p(X)},
\]
as desired. This completes the proof of \eqref{eq:Uhlenbeck_Chern_corollary_4-3_uA-Gamma_W1p_bound} in Theorem \ref{thm:Uhlenbeck_Chern_corollary_4-3} under the additional assumption \eqref{eq:Hodge_Laplacian_flat_connection_vanishing_kernel}, if $p$ obeys $d/2 \leq p < d$.

If $p = d$, one applies the argument for the case $d/2 \leq p < d$ \mutatis but using the Sobolev embedding $W^{1,d}(X) \subset L^s(X)$ for $1 \leq s < \infty$ and the H\"older inequality with $r$ in the range $2d \leq r < \infty$ defined by $1/d  = 1/q + 1/r$, where $d < q \leq 2d$.

If $d < p < \infty$, one again applies the argument for the case $d/2 \leq p < d$ \mutatis but now using the Sobolev embedding $W^{1,p}(X) \subset L^\infty(X)$ and the H\"older inequality with $r$ in the range $2p \leq r \leq \infty$ defined by $1/p  = 1/q + 1/r$, where $p \leq q \leq 2p$.
\end{proof}

\begin{rmk}[On the assumption \eqref{eq:Hodge_Laplacian_flat_connection_vanishing_kernel} that $\Delta_\Gamma$ has zero kernel]
In general, we do not know that $\Ker \Delta_\Gamma \cap \Omega^1(X; \ad P) = 0$ unless we assume a technical hypothesis for $P$ that this kernel vanishing condition holds for all flat connections, $\Gamma$, on $P$ or else assume a topological hypothesis for $X$, such as $\pi_1(X)=\{1\}$, so $P \cong X\times G$ if and only if $P$ is flat \cite[Theorem 2.2.1]{DK}. In the latter case, $\Gamma$ is gauge-equivalent to the product connection and $\Ker \Delta_\Gamma \cong H^1(X;\RR)$, so an additional hypothesis for $X$ that $H^1(X;\ZZ) = 0$ would ensure the kernel vanishing condition \eqref{eq:Hodge_Laplacian_flat_connection_vanishing_kernel}.
\end{rmk}

\subsection{Existence of a flat connection when the curvature of the given connection is $L^\infty$-small}
\label{subsec:Yang_existence_flat_connection}
In \cite{Yang_2003pjm}, Yang observed that if one assumes that the given connection, $A$ on $P$, is smooth and has $L^\infty$-small curvature (rather than just $L^p$-small curvature for $p > d/2$), then one can give a more elementary (albeit lengthier) proof of existence of a flat connection $\Gamma$ in Theorem \ref{thm:Uhlenbeck_Chern_corollary_4-3}. Yang's results (Theorem \ref{thm:Yang_7}, Corollary \ref{cor:Yang_1}, and Theorem \ref{thm:Yang_3}) are not required elsewhere in our article.

\begin{thm}[Existence of a flat connection]
\label{thm:Yang_7}
(See \cite[Theorem 7]{Yang_2003pjm}.)
Let $G$ be a compact Lie group and $X$ be a compact, smooth manifold of dimension $d\geq 2$ endowed with a Riemannian metric, $g$. Let $\sU = \{U_\alpha\}_{\alpha\in\sI}$ be a finite open cover of $X$ and $g_{\alpha\beta}: U_{\alpha\beta} = U_\alpha \cap U_\beta \to G$ be a set of smooth transition functions, with respect to $\sU$, for a smooth principal $G$-bundle over $X$. Then there are constants $\eps = \eps(g,G,\sU) \in (0,1]$ and $C = C(g,G,\sU) \in [1,\infty)$, such that if
\[
\delta := \sup_{\begin{subarray}{c} x,y \in U_{\alpha\beta} \\ \alpha, \beta \in \sI \end{subarray}}
|g_{\alpha\beta}(x) - g_{\alpha\beta}(y)| < \eps,
\]
then there exist $\sV = \{V_\alpha\}_{\alpha\in \sI}$, a finite open cover of $X$ with $V_\alpha \subset U_\alpha$, a collection of \emph{constant} transition functions, $g_{\alpha\beta}^0: V_{\alpha\beta} = V_\alpha \cap V_\beta \to G$, and a collection of smooth functions, $\rho_\alpha : V_\alpha \to G$, such that
\[
\rho_\alpha g_{\alpha\beta} \rho_\beta^{-1} = g_{\alpha\beta}^0 \quad\text{on } V_\alpha\cap V_\beta,
\]
and
\[
\sup_{\begin{subarray}{c} x \in V_\alpha \\ \alpha\in \sI \end{subarray}}
\left|\rho_\alpha(x) - \id\right| < C\delta.
\]
In particular, the bundle defined by $\{g_{\alpha\beta}\}$ is isomorphic to the flat bundle defined by $\{g_{\alpha\beta}^0\}$.
\end{thm}

\begin{cor}[Existence of a flat connection]
\label{cor:Yang_1}
(See \cite[Corollary 1]{Yang_2003pjm}.)
Let $G$ be a compact Lie group and $X$ be a compact, smooth manifold of dimension $d\geq 2$ endowed with a Riemannian metric, $g$. Let $\sU = \{U_\alpha\}_{\alpha \in \sI}$ be a finite open cover of $X$ such that any two points $x, y$ in a nonempty intersection $U_\alpha \cap U_\beta$ can be connected by a $C^1$ curve within $U_\alpha \cap U_\beta$ with length $\leq l$, a uniform constant, and let $\{g_{\alpha\beta}\}$ be a set of smooth transition functions, with respect to $\sU$, for a smooth principal $G$-bundle over $X$. Then there are constants $\eps = \eps(g,G,l, \sU) \in (0, 1]$ and $C = C(g,G,\sU) \in [1, \infty)$ with the following significance. If
\[
\delta := \sup_{\begin{subarray}{c} x\in U_\alpha\cap U_\beta, \\ \alpha, \beta \in \sI \end{subarray}}
\left|\nabla g_{\alpha\beta}(x)\right| \leq \eps,
\]
then we have the same conclusions as in Theorem \ref{thm:Yang_7}. In particular, the bundle defined by $\{g_{\alpha\beta}\}$ is smoothly isomorphic to a flat bundle.
\end{cor}

Corollary \ref{cor:Yang_1} in turn leads to the following $L^\infty$ analogue of the existence of a flat connection in Theorem \ref{thm:Uhlenbeck_Chern_corollary_4-3} given a hypothesis \eqref{eq:Curvature_Lp_small} on smallness of the $L^p$ norm of the curvature.

\begin{thm}[Existence of a flat connection on a principal bundle supporting a $C^\infty$ connection with $L^\infty$-small curvature]
\label{thm:Yang_3}
(See \cite[Theorem 3 and Corollary 2]{Yang_2003pjm}.)
Let $X$ be a compact, smooth manifold of dimension $d\geq 2$ and endowed with a Riemannian metric, $g$, and $G$ be a compact Lie group. Then there is a constant, $\eps=\eps(d, g, G) \in (0,1]$, with the following significance. If $A$ is a $C^\infty$ connection on a $C^\infty$ principal $G$-bundle, $P$, over $X$ such that
\[
\|F_A\|_{L^\infty(X)} \leq \eps,
\]
then $P$ is $C^\infty$ isomorphic to a flat principal $G$-bundle.
\end{thm}


\begin{rmk}[Comparison with Theorem \ref{thm:Uhlenbeck_Chern_corollary_4-3}]
\label{rmk:Compare_Yang_Theorem_3_and_Uhlenbeck_Chern_Corollary_4-3}
Yang had suggested in \cite[Section 1]{Yang_2003pjm} that it might be possible to relax the $L^\infty$ condition on $F_A$ in his Theorem \ref{thm:Yang_3} to an $L^p$ condition, for some $p<\infty$. In fact, such a conclusion (for any $p > d/2$) had been proved as one part of the Theorem \ref{thm:Uhlenbeck_Chern_corollary_4-3} due to Uhlenbeck. However, the proof of Theorem \ref{thm:Yang_3} is more elementary than that of Theorem \ref{thm:Uhlenbeck_Chern_corollary_4-3} since it does not rely (explicitly or implicitly) on the existence of local Coulomb gauges and Sobolev estimates for local connection one-forms, $a = A-\Theta$, over a ball (in Coulomb gauge with respect to the product connection, $\Theta$) in terms of the $L^p$ norm of the curvature, $F_A$. Instead, Yang only uses the simple $L^\infty$ estimate for $A$ in radial gauge in terms of the $L^\infty$ estimate for $F_A$ \cite[Lemma 2.1]{UhlRem} in his proof of Theorem \ref{thm:Yang_3}.
\end{rmk}

\begin{rmk}[On Theorem \ref{thm:Yang_3} for vector bundles]
\label{rmk:Yang_Theorem_3_for_vector_bundles}
For consistency with the rest of our article, we have converted \cite[Theorem 3 and Corollary 2]{Yang_2003pjm} to the equivalent setting of a principal $G$-bundle rather than its original setting of a vector bundle $E$ with compact Lie structure group, $G$, and an orthogonal representation, $G\hookrightarrow \SO(l)$, for some integer $l\geq 2$ as in Yang \cite[Section 1]{Yang_2003pjm}.
\end{rmk}

\subsection{A priori interior estimate for the curvature of a Yang--Mills connection in dimension two}
\label{subsec:Apriori_estimate_curvature_Yang-Mills_connection_d_is_2}
As we noted in Section \ref{subsec:Apriori_estimate_curvature_Yang-Mills_connection}, Theorem \ref{thm:Uhlenbeck_3-5} does not cover the case $d=2$ but the forthcoming Lemma \ref{lem:Smith_theorem_4-1} provides an \apriori estimate that is adequate for the purposes of our proof of Theorem \ref{mainthm:Main} in the case $d=2$. Recall from \cite[p. 33]{UhlLp} that if $A$ is a $W^{1,p}$ Yang--Mills connection (for $p \in (1,\infty)$ obeying $p\geq d/2$), then $A$ is gauge-equivalent to a smooth Yang--Mills connection. The constant $C$ appearing in the statement of Lemma \ref{lem:Smith_theorem_4-1} can be computed explicitly in terms of Sobolev embedding norms for a ball of radius $r$ in $\RR^2$ (see \cite{AdamsFournier}) but we shall not require that refinement in this article.

\begin{lem}[\Apriori estimate for the curvature of a Yang--Mills connection in dimension two]
\label{lem:Smith_theorem_4-1}
(Compare \cite[Theorem 4.1]{Smith_1990}.)
If $p \in [1,\infty)$ and $r>0$ are constants, then there is a constant, $C=C(p,r) \in [1,\infty)$, with the following significance. Let $G$ be a compact Lie group and $A$ be a Yang--Mills connection with respect to the standard Euclidean metric on $B_r\times G$, where $B_r \subset \RR^2$ is the open ball with center at the origin in $\RR^2$ and radius $r>0$. If $F_A \in L^1(B_r;\Lambda^2\otimes\fg)$, then
\begin{equation}
\label{thm:Uhlenbeck_3-5_Lp_norm_FA_leq_constant_L1_norm_FA_ball}
\|F_A\|_{L^p(B_r)} \leq C\|F_A\|_{L^1(B_r)}.
\end{equation}
\end{lem}

\begin{proof}
We adapt the proof of  \cite[Theorem 4.1]{Smith_1990}. Noting that $*F_A \in \Omega^0(B_r;\fg)$ when $d=2$, the Kato Inequality \cite[Equation (6.20)]{FU} and the Yang--Mills equation for $A$ (see Section \ref{sec:Preliminaries}) imply that
\begin{equation}
\label{eq:Kato-Yang-Mills}
|d|F_A|| = |d|*F_A|| \leq |d_A*F_A| = |d_A^*F_A| = 0 \quad\text{on } B_r.
\end{equation}
By hypothesis, $|F_A| \in L^1(B_r)$ and clearly $\nabla |F_A| \in L^1(B_r)$, so $|F_A| \in W^{1,1}(B_r)$. The Sobolev Embedding \cite[Theorem 4.12, Part C]{AdamsFournier} (since $1^*=2$ for $d=2$) ensures that $W^{1,1}(B_r) \subset L^2(B_r)$ and so $|F_A| \in L^2(B_r)$. But then $|F_A| \in W^{1,2}(B_r)$ since $\nabla |F_A| \in L^2(B_r)$. The Sobolev Embedding \cite[Theorem 4.12, Part B]{AdamsFournier} (for $d=2$) implies that $W^{1,2}(B_r) \subset L^p(B_r)$ for any $p \in [1,\infty)$. We now combine these observations to give
\begin{align*}
\|F_A\|_{L^p(B_r)} &\leq C\|F_A\|_{W^{1,2}(B_r)}  \quad\text{(by \cite[Theorem 4.12, Part B]{AdamsFournier})}
\\
&= C\|F_A\|_{L^2(B_r)} \quad\text{(by \eqref{eq:Kato-Yang-Mills})}
\\
&\leq C\|F_A\|_{W^{1,1}(B_r)} \quad\text{(by \cite[Theorem 4.12, Part C]{AdamsFournier})}
\\
&= C\|F_A\|_{L^1(B_r)} \quad\text{(by \eqref{eq:Kato-Yang-Mills})},
\end{align*}
as desired.
\end{proof}

Lemma \ref{lem:Smith_theorem_4-1} serves as a replacement for Theorem \ref{thm:Uhlenbeck_3-5} when $d=2$ and in our application, we use the following immediate corollary and analogue of Corollary \ref{cor:Uhlenbeck_3-5_manifold}.

\begin{cor}[\Apriori estimate for the curvature of a Yang--Mills connection over a closed two-dimensional manifold]
\label{cor:Smith_theorem_4-1_manifold}
Let $X$ be a closed, smooth, two-dimensional manifold endowed with a Riemannian metric, $g$, and $p \in [1,\infty)$ be a constant. Then there is a constant, $K_p=K_p(g,p) \in [1,\infty)$, with the following significance. Let $G$ be a compact Lie group and $A$ be a smooth Yang--Mills connection with respect to the metric, $g$, on a smooth principal $G$-bundle $P$ over $X$. Then
\begin{equation}
\label{eq:Uhlenbeck_3-5_Lp_norm_FA_leq_constant_L1_norm_FA_manifold}
\|F_A\|_{L^p(X)} \leq K_p\|F_A\|_{L^1(X)}.
\end{equation}
\end{cor}

\section{Corrigenda}
\label{sec:Corrections}
In this Appendix, we give a cumulative list of the mathematical corrections to the previously published version \cite{Feehan_yangmillsenergygapflat_aim} of our article that are provided in this version. We provide a fuller discussion of these corrections in \cite{Feehan_yangmillsenergygapflat_corrigendum}. Corrections to small typographical errors are not noted. 

The most significant correction involves replacement of the role of \cite[Corollary 4.3]{UhlChern} due to Uhlenbeck by that of \cite[Theorems 1 and 9]{Feehan_nonlinear_uhlenbeck_estimate} due to the author. We compare Theorem \ref{thm:Uhlenbeck_Chern_corollary_4-3} and \cite[Corollary 4.3]{UhlChern} in Section \ref{sec:Uhlenbeck_approach_existence_flat_connection_and_apriori_estimate}. The replacement of appeals to \cite[Corollary 4.3]{UhlChern} by appeals to Theorem \ref{thm:Uhlenbeck_Chern_corollary_4-3} in our proofs of Corollary \ref{cor:Uhlenbeck_Chern_A_Yang-Mills_FA_Ldover2_small} and Theorem \ref{mainthm:Main} in Section \ref{sec:Completion_proof_main_theorem} require only slight changes to those proofs and those were the only applications of \cite[Corollary 4.3]{UhlChern}. We list below the remaining corrections to \cite{Feehan_yangmillsenergygapflat_aim}:

\begin{itemize}
  \item The new hypothesis $d\geq 3$ corrects the previous hypothesis $d\geq 2$ in Theorem \ref{thm:Uhlenbeck_Lp_1-3} and Theorem \ref{thm:Uhlenbeck_3-5}.
  \item Corollaries \ref{cor:Uhlenbeck_theorem_1-3_p_lessthan_dover2} and \ref{cor:Uhlenbeck_theorem_1-3_p_geq_d} correct and replace our our \cite[Corollary 4.4]{Feehan_yangmillsenergygapflat_aim}, which should have included the restriction $p>1$ when $d=2$.
  \item An explanation for the correction of the hypothesis $d\geq 2$ to $d\geq 3$ in Theorem \ref{thm:Uhlenbeck_3-5} is added in the last paragraph of Section \ref{subsec:Apriori_estimate_curvature_Yang-Mills_connection}
  \item The proof of Corollary \ref{cor:Uhlenbeck_Chern_A_Yang-Mills_FA_Ldover2_small} is slightly modified to correct for the case $d=2$.
  \item Theorem \ref{thm:Uhlenbeck_3-5} and Corollary \ref{cor:Uhlenbeck_3-5_manifold} do not cover the case $d=2$ but the added Lemma \ref{lem:Smith_theorem_4-1} and Corollary \ref{cor:Smith_theorem_4-1_manifold} in the new Section \ref{subsec:Apriori_estimate_curvature_Yang-Mills_connection_d_is_2} provide an alternative \apriori estimate covering the case $d=2$ that is sufficient for the purposes of this article.
\end{itemize}

%
%

\bibliography{master,mfpde}

\def\cprime{$'$} \def\cprime{$'$}
  \def\ocirc#1{\ifmmode\setbox0=\hbox{$#1$}\dimen0=\ht0 \advance\dimen0
  by1pt\rlap{\hbox to\wd0{\hss\raise\dimen0
  \hbox{\hskip.2em$\scriptscriptstyle\circ$}\hss}}#1\else {\accent"17 #1}\fi}
  \def\cprime{$'$} \def\cprime{$'$} \def\cprime{$'$} \def\cprime{$'$}
  \def\polhk#1{\setbox0=\hbox{#1}{\ooalign{\hidewidth
  \lower1.5ex\hbox{`}\hidewidth\crcr\unhbox0}}} \def\cprime{$'$}
  \def\cprime{$'$} \def\cprime{$'$}
  \def\lfhook#1{\setbox0=\hbox{#1}{\ooalign{\hidewidth
  \lower1.5ex\hbox{'}\hidewidth\crcr\unhbox0}}} \def\cprime{$'$}
  \def\cprime{$'$} \def\cprime{$'$} \def\cprime{$'$} \def\cprime{$'$}
\providecommand{\bysame}{\leavevmode\hbox to3em{\hrulefill}\thinspace}
\providecommand{\MR}{\relax\ifhmode\unskip\space\fi MR }
\providecommand{\MRhref}[2]{%
  \href{http://www.ams.org/mathscinet-getitem?mr=#1}{#2}
}
\providecommand{\href}[2]{#2}
\begin{thebibliography}{10}

\bibitem{AMR}
R.~Abraham, J.~E. Marsden, and T.~Ratiu, \emph{Manifolds, tensor analysis, and
  applications}, second ed., Springer, New York, 1988. \MR{960687 (89f:58001)}

\bibitem{AdamsFournier}
R.~A. Adams and J.~J.~F. Fournier, \emph{Sobolev spaces}, second ed.,
  Elsevier/Academic Press, Amsterdam, 2003. \MR{2424078 (2009e:46025)}

\bibitem{Agaoka_1995}
Y.~Agaoka, \emph{An example of almost flat affine connections on the
  three-dimensional sphere}, Proc. Amer. Math. Soc. \textbf{123} (1995),
  3519--3521. \MR{1283536 (96a:53033)}

\bibitem{Atiyah_1957}
M.~F. Atiyah, \emph{Complex analytic connections in fibre bundles}, Trans.
  Amer. Math. Soc. \textbf{85} (1957), 181--207. \MR{0086359 (19,172c)}

\bibitem{Auslander_Markus_1955}
L.~Auslander and L.~Markus, \emph{Holonomy of flat affinely connected
  manifolds}, Ann. of Math. (2) \textbf{62} (1955), 139--151. \MR{0072518
  (17,298b)}

\bibitem{Azad_Biswas_2003}
H.~Azad and I.~Biswas, \emph{On holomorphic principal bundles over a compact
  {R}iemann surface admitting a flat connection. {II}}, Bull. London Math. Soc.
  \textbf{35} (2003), 440--444. \MR{1978996 (2004d:14043)}

\bibitem{Bernard_Riviere_2014}
Y.~Bernard and T.~Rivi{\`e}re, \emph{Energy quantization for {W}illmore
  surfaces and applications}, Ann. of Math. (2) \textbf{180} (2014), 87--136.
  \MR{3194812}

\bibitem{Biswas_Subramanian_2004}
I.~Biswas and S.~Subramanian, \emph{Flat holomorphic connections on principal
  bundles over a projective manifold}, Trans. Amer. Math. Soc. \textbf{356}
  (2004), 3995--4018. \MR{2058516 (2005d:14058)}

\bibitem{Bourguignon_Karcher_1978}
J-P. Bourguignon and H.~Karcher, \emph{Curvature operators: pinching estimates
  and geometric examples}, Ann. Sci. \'Ecole Norm. Sup. (4) \textbf{11} (1978),
  71--92. \MR{0493867 (58 \#12829)}

\bibitem{Bourguignon_Lawson_1981}
J-P. Bourguignon and H.~B. Lawson, Jr., \emph{Stability and isolation phenomena
  for {Y}ang--{M}ills fields}, Comm. Math. Phys. \textbf{79} (1981), 189--230.
  \MR{612248 (82g:58026)}

\bibitem{Bourguignon_Lawson_Simons_1979}
J-P. Bourguignon, H.~B. Lawson, Jr., and J.~Simons, \emph{Stability and gap
  phenomena for {Y}ang--{M}ills fields}, Proc. Nat. Acad. Sci. U.S.A.
  \textbf{76} (1979), 1550--1553. \MR{526178 (80h:53028)}

\bibitem{Chavel}
I.~Chavel, \emph{Eigenvalues in {R}iemannian geometry}, Pure and Applied
  Mathematics, vol. 115, Academic Press, Inc., Orlando, FL, 1984, Including a
  chapter by Burton Randol, With an appendix by Jozef Dodziuk. \MR{768584}

\bibitem{Dodziuk_Min-Oo_1982}
J.~Dodziuk and M.~Min-Oo, \emph{An {$L_{2}$}-isolation theorem for
  {Y}ang--{M}ills fields over complete manifolds}, Compositio Math. \textbf{47}
  (1982), 165--169. \MR{677018 (84b:53033b)}

\bibitem{DK}
S.~K. Donaldson and P.~B. Kronheimer, \emph{The geometry of four-manifolds},
  Oxford University Press, New York, 1990.

\bibitem{Feehan_yangmillsenergygapflat_corrigendum}
P.~M.~N. Feehan, \emph{Corrigendeum to ``{E}nergy gap for {Y}ang--{M}ills
  connections, {II}: {A}rbitrary closed {R}iemannian manifolds''}, preprint,
  July 15, 2019.

\bibitem{Feehan_yangmillsenergygapflat_arxiv_v7}
\bysame, \emph{Energy gap for {Y}ang--{M}ills connections, {II}: {A}rbitrary
  closed {R}iemannian manifolds}, arXiv:1502.00668v7. This version incorporates
  corrections to minor errors noticed since publication online (May 25, 2017)
  in Advances in Mathematics \textbf{312} (2017), pp. 547--587, at
  \url{https://doi.org/10.1016/j.aim.2017.03.023}.

\bibitem{Feehan_yang_mills_gradient_flow_v4}
\bysame, \emph{Global existence and convergence of solutions to gradient
  systems and applications to {Y}ang--{M}ills gradient flow}, submitted to a
  refereed monograph series on September 4, 2014, arXiv:1409.1525v4, xx+475
  pages.

\bibitem{Feehan_nonlinear_uhlenbeck_estimate}
\bysame, \emph{Morse theory for the {Y}ang--{M}ills energy function near flat
  connections}, arXiv:1906.03954.

\bibitem{Feehan_lojasiewicz_inequality_ground_state_v1}
\bysame, \emph{Optimal {{\L}}ojasiewicz--{S}imon inequalities and
  {M}orse--{B}ott {Y}ang--{M}ills energy functions}, submitted to a refereed
  journal on June 28, 2017, arXiv:1706.09349v1.

\bibitem{FeehanSlice}
\bysame, \emph{Critical-exponent {S}obolev norms and the slice theorem for the
  quotient space of connections}, Pacific J. Math. \textbf{200} (2001), no.~1,
  71--118, arXiv:dg-ga/9711004. \MR{1863408}

\bibitem{Feehan_yangmillsenergygap}
\bysame, \emph{Energy gap for {Y}ang--{M}ills connections, {I}:
  Four-dimensional closed {R}iemannian manifolds}, Adv. Math. \textbf{296}
  (2016), 55--84, arXiv:1412.4114. \MR{3490762}

\bibitem{Feehan_yangmillsenergygapflat_aim}
\bysame, \emph{Energy gap for {Y}ang--{M}ills connections, {II}: {A}rbitrary
  closed {R}iemannian manifolds}, Adv. Math. \textbf{312} (2017), 547--587.
  \MR{3635819}

\bibitem{Feehan_Maridakis_Lojasiewicz-Simon_coupled_Yang-Mills_v4}
P.~M.~N. Feehan and M.~Maridakis, \emph{{{\L}}ojasiewicz--{S}imon gradient
  inequalities for coupled {Y}ang--{M}ills energy functions}, Memoirs of the
  American Mathematical Society, American Mathematical Society, Providence, RI,
  in press, arXiv:1510.03815v4.

\bibitem{FU}
D.~S. Freed and K.~K. Uhlenbeck, \emph{Instantons and four-manifolds}, second
  ed., Mathematical Sciences Research Institute Publications, vol.~1, Springer,
  New York, 1991. \MR{1081321 (91i:57019)}

\bibitem{FrM}
R.~Friedman and John~W. Morgan, \emph{Smooth four-manifolds and complex
  surfaces}, Ergebnisse der Mathematik und ihrer Grenzgebiete (3) [Results in
  Mathematics and Related Areas (3)], vol.~27, Springer--Verlag, Berlin, 1994.
  \MR{1288304}

\bibitem{Fukaya_9-25-2018}
Kenji Fukaya, \emph{personal communication}, September 25, 2018.

\bibitem{Fukaya_1998}
\bysame, \emph{Anti-self-dual equation on {$4$}-manifolds with degenerate
  metric}, Geom. Funct. Anal. \textbf{8} (1998), no.~3, 466--528. \MR{1631255}

\bibitem{GallotMeyer}
S.~Gallot and D.~Meyer, \emph{Op\'erateur de courbure et laplacien des formes
  diff\'erentielles d'une vari\'et\'e riemannienne}, J. Math. Pures Appl. (9)
  \textbf{54} (1975), 259--284. \MR{0454884 (56 \#13128)}

\bibitem{Gerhardt_2010}
C.~Gerhardt, \emph{An energy gap for {Y}ang--{M}ills connections}, Comm. Math.
  Phys. \textbf{298} (2010), 515--522. \MR{2669447 (2012d:58019)}

\bibitem{GilbargTrudinger}
D.~Gilbarg and N.~S. Trudinger, \emph{Elliptic partial differential equations
  of second order}, second ed., Grundlehren der Mathematischen Wissenschaften
  [Fundamental Principles of Mathematical Sciences], vol. 224, Springer-Verlag,
  Berlin, 1983. \MR{737190}

\bibitem{GroisserParkerSphere}
D.~Groisser and Thomas~H. Parker, \emph{The {R}iemannian geometry of the
  {Y}ang--{M}ills moduli space}, Comm. Math. Phys. \textbf{112} (1987),
  663--689. \MR{910586 (89b:58024)}

\bibitem{Hamilton_1986}
R.~S. Hamilton, \emph{Four-manifolds with positive curvature operator}, J.
  Differential Geom. \textbf{24} (1986), no.~2, 153--179. \MR{862046}

\bibitem{Huang_2006}
S.-Z. Huang, \emph{Gradient inequalities}, Mathematical Surveys and Monographs,
  vol. 126, American Mathematical Society, Providence, RI, 2006. \MR{2226672
  (2007b:35035)}

\bibitem{Huang_2015arxiv}
T.~Huang, \emph{Energy gap phenomena for {Y}ang--{M}ills connections},
  arXiv:1502.03198.

\bibitem{Huang_2017}
Teng Huang, \emph{A proof of energy gap for {Y}ang-{M}ills connections}, C. R.
  Math. Acad. Sci. Paris \textbf{355} (2017), no.~8, 910--913,
  arXiv:1704.02772. \MR{3693514}

\bibitem{Kobayashi}
S.~Kobayashi, \emph{Differential geometry of complex vector bundles},
  Publications of the Mathematical Society of Japan, vol.~15, Princeton
  University Press, Princeton, NJ, 1987, Kan{\^o} Memorial Lectures, 5.
  \MR{909698 (89e:53100)}

\bibitem{Kobayashi_Nomizu_v1}
S.~Kobayashi and K.~Nomizu, \emph{Foundations of differential geometry. {V}ol
  {I}}, Interscience Publishers, a division of John Wiley \& Sons, New
  York-London, 1963. \MR{0152974 (27 \#2945)}

\bibitem{Lawson}
H.~B. Lawson, Jr., \emph{The theory of gauge fields in four dimensions}, CBMS
  Regional Conference Series in Mathematics, vol.~58, Published for the
  Conference Board of the Mathematical Sciences, Washington, DC; by the
  American Mathematical Society, Providence, RI, 1985. \MR{799712}

\bibitem{Lockhart_McOwen_1985}
R.~B. Lockhart and R.~C. McOwen, \emph{Elliptic differential operators on
  noncompact manifolds}, Ann. Scuola Norm. Sup. Pisa Cl. Sci. (4) \textbf{12}
  (1985), 409--447. \MR{837256 (87k:58266)}

\bibitem{Lojasiewicz_1963}
S.~{\L}ojasiewicz, \emph{Une propri\'et\'e topologique des sous-ensembles
  analytiques r\'eels}, Les \'{E}quations aux {D}\'eriv\'ees {P}artielles
  ({P}aris, 1962), \'Editions du Centre National de la Recherche Scientifique,
  Paris, 1963, pp.~87--89. \MR{0160856 (28 \#4066)}

\bibitem{Lojasiewicz_1965}
\bysame, \emph{Ensembles semi-analytiques},  (1965), Publ. Inst. Hautes Etudes
  Sci., Bures-sur-Yvette. LaTeX version by M. Coste, August 29, 2006 based on
  mimeographed course notes by S. {\L}ojasiewicz, available at
  \url{perso.univ-rennes1.fr/michel.coste/Lojasiewicz.pdf}.

\bibitem{Lojasiewicz_1993}
\bysame, \emph{Sur la g\'eom\'etrie semi- et sous-analytique}, Ann. Inst.
  Fourier (Grenoble) \textbf{43} (1993), 1575--1595. \MR{1275210 (96c:32007)}

\bibitem{Matsushima_Okamoto_1979}
H.~Matsushima and K.~Okamoto, \emph{Nonexistence of torsion free flat
  connections on a real semisimple {L}ie group}, Hiroshima Math. J. \textbf{9}
  (1979), 59--60. \MR{529327 (80f:53026)}

\bibitem{Milnor_1958cmh}
J.~W. Milnor, \emph{On the existence of a connection with curvature zero},
  Comment. Math. Helv. \textbf{32} (1958), 215--223. \MR{0095518 (20 \#2020)}

\bibitem{Milnor_1976}
\bysame, \emph{Curvatures of left invariant metrics on {L}ie groups}, Advances
  in Math. \textbf{21} (1976), 293--329. \MR{0425012 (54 \#12970)}

\bibitem{MilnorStasheff}
J.~W. Milnor and J.~D. Stasheff, \emph{Characteristic classes}, Princeton
  University Press, Princeton, N. J.; University of Tokyo Press, Tokyo, 1974,
  Annals of Mathematics Studies, No. 76. \MR{0440554}

\bibitem{Min-Oo_1982}
M.~Min-Oo, \emph{An {$L_{2}$}-isolation theorem for {Y}ang--{M}ills fields},
  Compositio Math. \textbf{47} (1982), 153--163. \MR{0677017 (84b:53033a)}

\bibitem{MMR}
John~W. Morgan, Tomasz~S. Mrowka, and Daniel Ruberman, \emph{The {$L^2$}-moduli
  space and a vanishing theorem for {D}onaldson polynomial invariants},
  Monographs in Geometry and Topology, vol.~2, International Press, Cambridge,
  MA, 1994. \MR{1287851 (95h:57039)}

\bibitem{Mrowka_7-30-2018}
Tomasz~S. Mrowka, \emph{personal communication}, July 30, 2018.

\bibitem{Nakajima_1987}
H.~Nakajima, \emph{Removable singularities for {Y}ang--{M}ills connections in
  higher dimensions}, J. Fac. Sci. Univ. Tokyo Sect. IA Math. \textbf{34}
  (1987), 299--307. \MR{914024 (89d:58029)}

\bibitem{Nishinou_2007}
T.~Nishinou, \emph{Global gauge fixing for connections with small curvature on
  {$T^2$}}, Internat. J. Math. \textbf{18} (2007), no.~2, 165--177.
  \MR{2307419}

\bibitem{ParkerGauge}
Thomas~H. Parker, \emph{Gauge theories on four-dimensional {R}iemannian
  manifolds}, Comm. Math. Phys. \textbf{85} (1982), 563--602. \MR{677998
  (84b:58036)}

\bibitem{Petersen_riemannian_geometry3}
P.~Petersen, \emph{Riemannian geometry}, third ed., Graduate Texts in
  Mathematics, vol. 171, Springer, Cham, 2016. \MR{3469435}

\bibitem{Price_1983}
P.~Price, \emph{A monotonicity formula for {Y}ang--{M}ills fields}, Manuscripta
  Math. \textbf{43} (1983), 131--166. \MR{707042 (84m:58033)}

\bibitem{Rade_1992}
J.~R\r{a}de, \emph{On the {Y}ang--{M}ills heat equation in two and three
  dimensions}, J. Reine Angew. Math. \textbf{431} (1992), 123--163. \MR{1179335
  (94a:58041)}

\bibitem{Shen_1982}
C.~L. Shen, \emph{The gap phenomena of {Y}ang--{M}ills fields over the complete
  manifold}, Math. Z. \textbf{180} (1982), 69--77. \MR{656222 (83k:53048)}

\bibitem{Sibner_1984}
L.~M. Sibner, \emph{Removable singularities of {Y}ang--{M}ills fields in {${\bf
  R}^{3}$}}, Compositio Math. \textbf{53} (1984), 91--104. \MR{762308
  (86c:58151)}

\bibitem{SibnerSibnerUhlenbeck}
L.~M. Sibner, R.~J. Sibner, and K.~K. Uhlenbeck, \emph{Solutions to
  {Y}ang--{M}ills equations that are not self-dual}, Proc. Nat. Acad. Sci.
  U.S.A. \textbf{86} (1989), 8610--8613. \MR{1023811 (90j:58032)}

\bibitem{Simon_1983}
L.~Simon, \emph{Asymptotics for a class of nonlinear evolution equations, with
  applications to geometric problems}, Ann. of Math. (2) \textbf{118} (1983),
  525--571. \MR{727703 (85b:58121)}

\bibitem{Smith_1990}
P.~D. Smith, \emph{Removable singularities for the {Y}ang--{M}ills-{H}iggs
  equations in two dimensions}, Ann. Inst. H. Poincar\'e Anal. Non Lin\'eaire
  \textbf{7} (1990), no.~6, 561--588. \MR{1079572}

\bibitem{TauL2}
Clifford~H. Taubes, \emph{{$L^2$} moduli spaces on 4-manifolds with cylindrical
  ends}, Monographs in Geometry and Topology, I, International Press,
  Cambridge, MA, 1993. \MR{1287854 (96b:58018)}

\bibitem{UhlLp}
K.~K. Uhlenbeck, \emph{Connections with {$L^{p}$} bounds on curvature}, Comm.
  Math. Phys. \textbf{83} (1982), 31--42. \MR{648356 (83e:53035)}

\bibitem{UhlRem}
\bysame, \emph{Removable singularities in {Y}ang--{M}ills fields}, Comm. Math.
  Phys. \textbf{83} (1982), 11--29. \MR{648355 (83e:53034)}

\bibitem{UhlChern}
\bysame, \emph{The {C}hern classes of {S}obolev connections}, Comm. Math. Phys.
  \textbf{101} (1985), 449--457. \MR{815194 (87f:58028)}

\bibitem{Wehrheim_2004}
Katrin Wehrheim, \emph{Uhlenbeck compactness}, EMS Series of Lectures in
  Mathematics, European Mathematical Society (EMS), Z\"urich, 2004. \MR{2030823
  (2004m:53045)}

\bibitem{Weil_1938}
A.~Weil, \emph{Generalization de fonctions abeliennes}, J. Math. Pures Appl.
  \textbf{17} (1938), 47--87.

\bibitem{Xin_1980}
Y.~L. Xin, \emph{Some results on stable harmonic maps}, Duke Math. J.
  \textbf{47} (1980), 609--613. \MR{587168 (81j:58041)}

\bibitem{Xin_1984}
\bysame, \emph{Remarks on gap phenomena for {Y}ang--{M}ills fields}, Sci.
  Sinica Ser. A \textbf{27} (1984), 936--942, arXiv:math/0203077. \MR{767626
  (86f:58041)}

\bibitem{Yang_2003pjm}
B.~Yang, \emph{Removable singularities for {Y}ang--{M}ills connections in
  higher dimensions}, Pacific J. Math. \textbf{209} (2003), 381--398.
  \MR{1978378 (2004b:53037)}

\end{thebibliography}
\bibliographystyle{amsplain}

\end{document}